\documentclass[11pt,a4paper,reqno]{amsart}
\usepackage{hyperref}
\usepackage{amsmath}
\usepackage{amssymb, amscd}
\usepackage[all, cmtip]{xy}
\usepackage{xcolor}
\usepackage{multirow} 
\usepackage{longtable}
\usepackage{array}

\newtheorem{theorem}{Theorem}[section]

\newtheorem{conjecture}[theorem]{Conjecture}

\newtheorem{lemma}[theorem]{Lemma}

\theoremstyle{definition}

{\sc}

{}
{}
{\sc}
\newtheorem{remark}{Remark}[section]

\def\ZZ{\mathbb{Z}}

\newcommand{\acknowledge}{\subsection*{Acknowledgments}}
\newcommand{\thismonth}{\ifcase\month\or
  January\or February\or March\or April\or May\or June\or
  July\or August\or September\or October\or November\or December\fi
  \space\number\year}

\makeatletter
\newcommand{\rssymb}[2]{\newcommand{#1}{{\mathrmsl{#2}}}}
\newcommand{\calsymb}[2]{\newcommand{#1}{{\mathcal{#2}}}}
\newcommand{\bbsymb}[2]{\newcommand{#1}{{\mathbb{#2}}}}
\newcommand{\lieoper}[2]{\newcommand{#1}{\mathop
  {\mathfrak{#2}\null}\nolimits}}
\newcommand{\oper}[3][n]{\newcommand{#2}{\mathop
  {\mathrm{#3}\null}\ifx n#1\nolimits\else\limits\fi}}
\newcommand{\rsoper}[3][n]{\newcommand{#2}{\mathop
  {\mathrmsl{#3}\null}\ifx n#1\nolimits\else\limits\fi}}
\bbsymb\C{C} \bbsymb\F{F} \bbsymb\HQ{H}\bbsymb\N{N} \bbsymb\Q{Q}
\bbsymb\R{R} \bbsymb\U{U} \bbsymb\V{V} \bbsymb\W{W} \bbsymb\Z{Z}
\bbsymb\bbf{F} \bbsymb\bbk{K} \bbsymb\bbi{I} \bbsymb\bbl{L}
\bbsymb\bbo{O} \bbsymb\bbj{J} \bbsymb\bby{Y} \bbsymb\bbp{P}
\bbsymb\bba{A}
\calsymb\cA{A} \calsymb\cB{B} \calsymb\cC{C} 
\calsymb\cM{M} \calsymb\cN{N} \calsymb\cO{O} \calsymb\cP{P}
\calsymb\cU{U} \calsymb\cV{V} \calsymb\cW{W} \calsymb\cX{X}
\calsymb\cY{Y} \calsymb\cZ{Z}
\renewcommand{\geq}{\geqslant} \renewcommand{\leq}{\leqslant}
\oper\End{End} \oper\Hom{Hom}                    
\oper\Sym{Sym} \oper\Skew{Skew}
\oper\Aut{Aut}                                   
\oper\GL{GL} \oper\SL{SL}\oper\Symp{Sp} \oper\CO{CO} \oper\On{O}
\oper\SO{SO} \oper\Pin{Pin} \oper\Spin{Spin} \oper\CU{CU}
\oper\Un{U} \oper\SU{SU} \oper\PSU{PSU} \rsoper\Diff{Diff}
\rsoper\SDiff{SDiff}
\lieoper\der{der}                                
\lieoper\gl{gl} \lieoper\sgl{sl}\lieoper\symp{sp} \lieoper\co{co}
\lieoper\so{so} \lieoper\spin{spin} \lieoper\cu{cu} \lieoper\un{u}
\lieoper\su{su} \rsoper\Vect{Vect} \rsoper\Ham{Ham}
\def\la#1{\hbox to #1pc{\leftarrowfill}}
\def\ra#1{\hbox to #1pc{\rightarrowfill}}

\newcommand{\norm}[2][]{|\mkern-2mu|#2|\mkern-2mu|
  _{\lower1pt\hbox{${}_{#1}$}}}
\newcommand{\Norm}[2][]{\bigl|\mkern-3mu\bigr|#2\bigr|\mkern-3mu\bigr|
  _{\lower1pt\hbox{${}_{#1}$}}}

\rsoper\dimn{dim}                           
\rsoper\grad{grad}                          
\rsoper\kernel{ker}\rsoper\image{im}        
\rsoper\alt{alt}   \rsoper\sym{sym}         
\rsoper\Ad{Ad}     \rsoper\ad{ad}           
\rsoper\CoAd{CoAd} \rsoper\coad{coad}       
\rsoper\trace{tr}  \rsoper\trfree{tf}       
\rsoper\detm{det}                           
\rsoper\Vol{Vol}                            
\rsoper\divg{div}                           
\rsoper\sign{sign}                          
\rssymb\iden{id}                            
\rssymb\vol{vol}                            
\oper\Imag{Im}\oper\Real{Re}                
\newcommand{\sd}{{\raise1pt\hbox{$\scriptscriptstyle +$}}}
\newcommand{\asd}{{\raise1pt\hbox{$\scriptscriptstyle -$}}}
\newcommand{\sdasd}{{\raise1pt\hbox{$\scriptscriptstyle\pm$}}}

\newcommand{\asdsd}{{\raise1pt\hbox{$\scriptscriptstyle\mp$}}}

\rsoper\scal{scal}
\def\kahl/{k\"ahler}
\def\Kahl/{K{\"a}hler}

\begin{document}
\title[Sasaki-Einstein 7-manifolds and Orlik's conjecture]
{Sasaki-Einstein 7-manifolds and Orlik's conjecture}
\author[J. Cuadros]{Jaime Cuadros Valle}
\author[J. Lope]{Joe Lope Vicente}

\address{Departamento de Ciencias, Secci\'on Matem\'aticas,
Pontificia Universidad Cat\'olica del Per\'u,
Apartado 1761, Lima 100, Per\'u}
\email{jcuadros@pucp.edu.pe}
\email{j.lope@pucp.edu.pe}

\date{\thismonth}

\begin{abstract} 
We apply Orlik's conjecture to study the homology groups of  certain 2-connected 7-manifolds admitting quasi-regular Sasaki-Einstein metrics, among them, we found 52 new examples of Sasaki-Einstein rational homology 7-spheres, extending the list given by Boyer, Galicki and Nakamaye in \cite{BGN2}. 
As a consequence, we establish the existence of new  families of  homotopy 9-spheres admitting positive Ricci curvature and determine the diffeomorphism type of them. We also improve  previous results given by Boyer \cite{Bo}  showing  new examples of  Sasaki-Einstein 2-connected 7-manifolds homeomorphic to  connected sums of $S^3\times S^4.$ Actually we show that manifolds of the form $\#k\left(S^{3} \times S^{4}\right)$ admit Sasaki-Einstein metrics for 22 different values of $k.$ All 
 these links arise as  Thom-Sebastiani sums of chain type singularities and cycle type singularities  where Orlik's conjecture  holds due to  a rather recent result by Hertling and Mase \cite{HM}.   
 
\end{abstract} 
\maketitle
\vspace{-2mm} 


\section{Introduction}

In the last 20 years, several techniques  have been developed 
to determine the class of manifolds that admit metrics of positive Ricci curvature. For  odd dimensions, Boyer, Galicki and collaborators established the abundance  of positive Einstein metrics, actually Sasaki-Einstein metrics. In dimension five there are quite remarkable results (see \cite{Ko}, \cite{BGK}, \cite{GMS}) that leads us to think it is conceivable to give a complete  classification there. 
For dimension 7, there are several important results, in particular  7-manifolds  arising as links of hypersurface singularities were studied intensively using as  framework the seminal work of Milnor in \cite{Mi} where $S^{3}$-bundles over $S^{4}$ with structure group $SO(4)$ were studied.   
Using the Gysin sequence for these fiber bundles one obtains  three classes of 7-manifolds determined by  the Euler class $e$:
\begin{enumerate}
 \item[a)] If $e=\pm 1$ then the fiber bundle has the same homology of the 7-sphere.
 \item[b)] If $|e|\geq 2$  then the fiber bundle is a rational homology 7-sphere.   
 \item[c)] If $e=0$  then the fiber bundle has the homology of $S^{3}\times S^4.$
 \end{enumerate}
Milnor showed that manifolds  situated on the class described in item $a)$ were homeomorphic to $S^{7}.$ Furthermore, he proved that some of these bundles are not diffeomorphic to $S^{7}$ and as a result of that  exhibited the first examples of exotic spheres.  

An interesting example of manifolds described in $b)$ is the Stiefel manifold of 2-frames in $\mathbb R^5$, $V_2(\mathbb R^5)$. 
It is known that  this manifold can be realized as a link of quadric in $\mathbb C^5,$ and that   $V_2(\mathbb R^5)$ is a rational homology 7-sphere and moreover admits regular Sasaki-Einstein metric \cite{BG1, BFGK}. This example was a source of new techniques that led to establish the existence of many  Sasaki-Einstein structures.  Actually in \cite{BG1},  a method for proving the existence of Sasaki-Einstein metrics  on  links of hypersurface singularities is described.  In a sequel of papers   \cite{BGK}, \cite{BGN1},\cite{BGKT}, \cite{Bo},  Boyer, Galicki and their collaborators showed the existence of Sasaki-Einstein structures on   exotic spheres,   2-connected rational  homology 7-spheres and  connected sums of $S^{3} \times S^{4}$  respectively.  

In general, the homeomorphism type of links are not easy to be determined. However, for certain cases,  one can calculate  the integral homology of the links  through formulas derived  by Milnor and Orlik  in \cite{MO} and Orlik in \cite{Or}.  In fact,  Boyer exhibited fourteen examples \cite{Bo} of Sasaki-Einstein  7-manifolds arising from links of isolated hypersurface singularities from elements of  the well-known list of 95 codimension one K3 surfaces (\cite{IF, Ch}).  For these,  the third integral homology group is completely determined. In \cite{Go} Gomez,  calculated the torsion for the third integral homology group explicitly for ten examples of Sasaki-Einstein links of chain type singularities (also known as Orlik polynomials).
 
In this paper we benefit from  a recent result by Hertling and Mase in \cite{HM} where they show that the Orlik conjecture is valid for  chain type singularities, cycle type singularities and Thom-Sebastiani sums of them, and we update results given in \cite{BGN2}, \cite{Bo} and \cite{Go}. 
 From the list of 1936  Sasaki-Einstein 7-manifolds realized as links  from the list given in \cite{BGN2, JK} we detect 1673 that are links of hypersurface singularities of these types. Thus, via Orlik's algorithm we calculate  the third homology group for this lot. 
Among them, we found  52 new examples of 2-connected rational homology spheres admitting Sasaki-Einstein metrics. We also found  124 new examples of 2-connected  Sasaki-Einstein 7-manifolds of the form $\#2k(S^3\times S^4)$, improving a result of Boyer's in \cite{Bo}. Six of these new examples  are links of quasi-smooth Fano 3-folds coming from Reids list of 95 weighted codimension 1 K3 surfaces \cite{IF}, \cite{Ch}, the rest  of them, are links taken  from the list given in \cite{JK}. 

In recent years Sasaki-Einstein geometry has been intensely studied, part of this interest comes from a string theory conjecture known as the AdS/CFT correspondence which, under certain circumstances, relates superconformal field theories and Sasaki-Einstein metrics in dimension five and seven \cite{XY}. Also, certain rational homology spheres can be used to construct positive Sasakian structures in  homotopy 9-spheres \cite{BGN3}. 
Thereby, it is important to have as many examples as possible of Sasaki-Einstein manifolds especially in dimensions five and seven.

This paper is organized as follows: in Section 2 we give some preliminaries on the topology of links of hypersurface singularities. In Section 3, as a consequence of the explicit calculations on the topology of new rational Sasaki-Einstein homology 7-spheres given in this paper,  we  update some results on homotopy 9-spheres admitting positive Ricci curvatures given in \cite{BGN3}. Then we present Table 1 (listing new examples of rational homology 7-spheres admitting Sasaki-Einstein metrics),  Table 2  (listing new examples  of 7-links homeomorphic to connected sums of $S^3\times S^4$ admitting Sasaki-Einstein metrics produced from the Johnson and Kollár list)  and Table 3 (listing new examples  of Sasaki-Einstein 7-links homeomorphic to connected sums of $S^3\times S^4$  produced from Cheltsov's list). In these three tables we list the weights, one quasihomogenous polynomial generating the link, the type of singularity,  the degree, the Milnor number and finally the third homology group.   In Section 4,  we give a link to the four codes implemented in Matlab, these codes determine whether or  not the links come from the admissible type of singularities where Orlik's conjecture is valid, and compute the homology groups of the links under discussion. we  also give links to three additional tables (listing 7-links with non-zero third Betti number and with torsion).

\acknowledge The first author thanks Ralph Gomez and Charles Boyer for useful conversations. Part of this article was prepared with the financial support from Pontificia Universidad Católica del Perú through project DGI PI0655.


\section{Preliminaries: Sasaki-Einstein metrics on links and Orlik's conjecture}

In this section we briefly review the Sasakian geometry of links of isolated hypersurface singularities defined by weighted homogeneous polynomials. We describe the explicit constructions of Sasaki-Einstein manifolds
given by Boyer and Galicki \cite{BBG}. Then we give some known facts on the topology of links of hypersurface singularites \cite{Mi}, \cite{MO} and state Orlik's conjecture.  We also set up a table with the necessary conditions to obtain links where this conjecture is known to be valid.

\subsection{\textit{Links and Sasaki-Einstein metrics}} 
Consider the weighted $\mathbb{C}^{*}$ action on $\mathbb{C}^{n+1}$ given by 
$$
\left(z_{0}, \ldots, z_{n}\right) \longmapsto\left(\lambda^{w_{0}} z_{0}, \ldots, \lambda^{w_{n}} z_{n}\right)
$$
where $w_{i}$ are the weights which are positive integers and $\lambda \in \mathbb{C}^{*}$. Let us denote the weight vector by  $\mathbf{w}=\left(w_{0}, \ldots, w_{n}\right).$ We assume
$
\operatorname{gcd}\left(w_{0}, \ldots, w_{n}\right)=1.
$
Recall that a polynomial $f \in \mathbb{C}\left[z_{0}, \ldots, z_{n}\right]$ is said to be a weighted homogeneous polynomial of degree $d$ and weight $\mathbf{w}=\left(w_{0}, \ldots, w_{n}\right)$ if for any $\lambda \in \mathbb{C}^{*}=\mathbb{C} \backslash\{0\}$, we have
$$
f\left(\lambda^{w_{0}} z_{0}, \ldots, \lambda^{w_{n}} z_{n}\right)=\lambda^{d} f\left(z_{0}, \ldots, z_{n}\right) .
$$

We are interested in those weighted homogeneous polynomials $f$ whose zero locus in $\mathbb{C}^{n+1}$ has only an isolated singularity at the origin. 
The link $L_{f}(\mathbf{w}, d)$  (sometimes written $L_f$) is defined by  $$L_{f}(\mathbf{w}, d)=f^{-1}(0) \cap S^{2 n+1},$$ where $S^{2 n+1}$ is the $(2 n+1)$-sphere in $\mathbb{C}^{n+1}$. By the Milnor fibration theorem \cite{Mi}, $L_{f}(\mathbf{w}, d)$ is a closed $(n-2)$-connected $(2n-1)$-manifold that bounds a parallelizable manifold with the homotopy type of a bouquet of $n$-spheres. Furthermore, $L_{f}(\mathbf{w}, d)$ admits a quasi-regular Sasaki structure in a natural way, see for instance  \cite{Ta}.  Moreover, if one considers the locally free $S^1$-action induced by the weighted $\mathbb{C}^{*}$ action on $f^{-1}(0)$ the quotient space of the link  $L_{f}(\mathbf{w}, d)$ by this action is the weighted hypersurface $\mathcal{Z}_{f},$ a K\"ahler orbifold. Actually, we have the following commutative diagram \cite{BBG}

\begin{equation*}
\begin{CD} 
 L_{f}(\mathbf{w}, d) @> {\qquad\qquad}>>  S^{2n+1}_{\bf w}\\
@VV{\pi}V  @VVV\\
{\mathcal Z}_{f}  @> {\qquad\qquad}>>  {\mathbb P}({\bf w}),
\end{CD}
\end{equation*}
where $S_{\mathbf{w}}^{2 n+1}$ denotes the unit sphere with a weighted Sasakian structure, $\mathbb{P}(\mathbf{w})$ is a weighted projective space coming from the quotient of $S_{\mathbf{w}}^{2 n+1}$ by a weighted circle action generated from the weighted Sasakian structure. The top horizontal arrow is a Sasakian embedding and the bottom arrow is Kählerian embedding and  the vertical arrows are orbifold Riemannian submersions. 

It follows from the orbifold adjuntion formula that the  link $L_f$ admits a positive Ricci curvature if the quotient orbifold $Z_{f}$ by the natural action $S^1$ is Fano, which is equivalent to 
\begin{equation}
|{\bf w}|-d_f >0,  
\end{equation}
Here $|\mathbf{w}|=\sum_{i=0}^m w_i$ denotes the norm of the weight vector $\mathbf{w}$ and $d_f$ is the degree of the polynomial $f$. Furthermore,  in \cite{BG1}, Boyer and Galicki found a method to obtain 2-connected Sasaki-Einstein 7-manifolds from the existence of  orbifold Fano Kähler-Einstein hypersurfaces $\mathcal{Z}_{f}$ in weighted projective 4-space $\mathbb{P}(\mathbf{w})$. Actually, they showed a  more general result:

\begin{theorem}
The link $L_{f}(\mathbf{w}, d)$  admits a Sasaki-Einstein structure if and only if the Fano orbifold 
$\mathcal{Z}_{f}$ admits a Kähler-Einstein orbifold metric of scalar curvature $4 n(n+1).$
\end{theorem}

In \cite{JK}, Johnson and Kollár give a list of 4442 quasi-smooth Fano 3 -folds $\mathcal{Z}$ anticanonically embedded in weighted projective 4-spaces $\mathbb{P}(\mathbf{w})$. Moreover, they show that 1936 of these 3 -folds admit Kähler-Einstein metrics. Thus, such Fano 3-folds give rise to Sasaki-Einstein metrics on smooth 7-manifolds  realized as links  of isolated hypersurface singularities defined by weighted homogenous polynomials. In \cite{BGN2} they extracted from this list 184 2-connected rational homology 7-spheres. They also determined the order of $H_{3}\left (L_{f}(\mathbf{w}, d), \mathbb{Z}\right)$. In \cite{Go}, Gomez used Orlik's conjecture to  calculate the homology of 10 elements of the list given in \cite{BGN2}, all the 7-manifolds found there are links of chain type singularities.

In this paper we completely determine the third homology group for 1673 
of 2-connected Sasaki-Einstein 7-manifolds from the list 1936 smooth 7-manifolds  realized as links  of isolated hypersurface singularities from the list given in \cite{BGN2}. Among them, we found  52 new examples of 2-connected rational homology spheres admitting Sasaki-Einstein metrics. 
We also found  124 new examples of 2-connected  Sasaki-Einstein 7-manifolds of the form $\#2k(S^3\times S^4)$, improving a result of Boyer in \cite{Bo}, see Theorem 3.1.  Of that lot, 118 come from the list given in \cite{JK}, the other 6 examples are links of quasi-smooth Fano 3-folds coming from Reids list of 95 weighted codimension 1 K3 surfaces, where all members of the list  but 4 admit K\"ahler-Einstein metrics, see \cite{Ch}.

\subsection{\textit{The Topology of Links and Orlik's Conjecture}} 

In this section we  review some classical results on the topology of links of quasi-smooth hypersurface singularities.  Recall that 
the Alexander  polynomial $\Delta_f(t)$ in \cite{Mi} associated to a link $L_f$ of dimension $2 n-1$ is the characteristic polynomial of the monodromy map $$h_*: H_{n}(F, \mathbb{Z}) \rightarrow H_{n}(F, \mathbb{Z})$$ induced by the $S_{\mathrm{w}}^1$-action on the Milnor fibre $F$. Therefore, 
$\Delta_f(t)=\operatorname{det}\left(t {\mathbb I}-h_*\right)$. 
Now both $F$ and its closure $\bar{F}$ are homotopy equivalent to a bouquet of $n$-spheres $S^n \vee \cdots \vee S^n,$ and the boundary of $\bar{F}$ is the link $L_f$, which is $(n-2)$-connected. The Betti numbers $b_{n-1}\left(L_f\right)=b_{n}\left(L_f\right)$ equal the number of factors of $(t-1)$ in $\Delta_f(t)$. The following are standard facts, see \cite{BBG}. 
\begin{enumerate}
\item $L_f$ is a rational homology sphere if and only if $\Delta_f(1) \neq 0.$
\item $L_f$  homotopy sphere if and only if $\Delta_f(1)= \pm 1.$ 
\item If  $L_f$ is a rational homology sphere, then the order of $H_{n-1}\left(L_f, \mathbb{Z}\right)$ equals $|\Delta(1)|.$ 
\end{enumerate}

There is a remarkable  theorem of Levine (see \cite{Mi}, page 69) that determines the diffeomorphism type for rational homology spheres. More preceisely, we have  
\begin{theorem} Let $L_f$ be homeomorphic to the $(2n-1)$-sphere for $n$ odd. 
 $L_f$ is diffeomorphic to  the standard sphere if $\Delta_f(-1) \equiv \pm 1(\bmod 8)$ and $L_f$ is diffeomorphic to the exotic Kervaire sphere if $\Delta_f(-1) \equiv \pm 3 (\bmod 8).$ 
 \end{theorem}

In the case that $f$ is a weighted homogeneous polynomial  there is an algorithm due to Milnor and Orlik \cite{MO} to calculate  the free part of $H_{n-1}\left(L_f, \mathbb{Z}\right).$ The authors associate to any monic polynomial $f$ with roots $\alpha_1, \ldots, \alpha_k \in \mathbb{C}^*$ its divisor
$$
\operatorname{div} f=\left\langle\alpha_1\right\rangle+\cdots+\left\langle\alpha_k\right\rangle
$$
as an element of the integral ring $\mathbb{Z}\left[\mathbb{C}^*\right].$ Let  $\Lambda_n=\operatorname{div}\left(t^n-1\right)$. 
Then the divisor of $\Delta_f(t)$ is given by
\begin{equation}
\operatorname{div} \Delta_f=\prod_{i=0}^n\left(\frac{\Lambda_{u_i}}{v_i}-1\right),
\end{equation}
where the $u_i's$  and $v_i's$ are given  terms of the degree $d$ of $f$ and the weight vector ${\bf w}=(w_0,\ldots w_n)$ by the equations 
\begin{equation}
u_i=\frac{d}{\operatorname{gcd}\left(d, w_i\right)}, \quad v_i=\frac{w_i}{\operatorname{gcd}\left(d, w_i\right)}.
\end{equation}

Using the relations $\Lambda_a \Lambda_b=\operatorname{gcd}(a, b) \Lambda_{\operatorname{lcm}(a, b)}$, Equation (2) takes the form
\begin{equation}
\operatorname{div} \Delta_f=(-1)^{n+1}+\sum a_j \Lambda_j,
\end{equation}

where $a_j \in \mathbb{Z}$ 
 and the sum is taken over the set of all least common multiples of all combinations of the $u_0, \ldots, u_n$. Then the Alexander polynomial has an alternative expression given by
$$
\Delta_f(t)=(t-1)^{(-1)^n} \prod_j\left(t^j-1\right)^{a_j},
$$
and

\begin{equation}
b_{n-1}\left(L_f\right)=(-1)^{n+1}+\sum_j a_j.
\end{equation}

Moreover, Milnor and Orlik gave and explicit formula to calculate the free part of $H_{n-1}\left(L_f, \mathbb{Z}\right):$ 
\begin{equation}
b_{n-1}\left(L_{f}\right)=\sum(-1)^{n+1-s} \frac{u_{i_{1}} \cdots u_{i_{s}}}{v_{i_{1}} \cdots v_{i_{s}} \operatorname{lcm}\left(u_{i_{1}}, \ldots, u_{i_{s}}\right)},
\end{equation}
where the sum is taken over all the $2^{n+1}$ subsets $\left\{i_{1}, \ldots, i_{s}\right\}$ of $\{0, \ldots, n\}$. In \cite{Or}, Orlik gave a conjecture which allows to determine the torsion of homology of the link solely in terms of the weight of $f.$

\begin{conjecture}[Orlik] 
Consider  $\left\{i_{1}, \ldots, i_{s}\right\} \subset\{0,1, \ldots, n\}$ the set of ordered set of $s$ indices, that is, $i_{1}<i_{2}<\cdots<i_{s}$. Let us denote by I its power set (consisting of all of the $2^{s}$ subsets of the set), and by $J$ the set of all proper subsets. Given a $(2 n+2)$-tuple $(\mathbf{u}, \mathbf{v})=\left(u_{0}, \ldots, u_{n}, v_{0}, \ldots, v_{n}\right)$ of integers, let us  define inductively a set of $2^{s}$ positive integers, one for each ordered element of $I$, as follows:
$$
c_{\emptyset}=\operatorname{gcd}\left(u_{0}, \ldots, u_{n}\right),
$$
and if $\left\{i_{1}, \ldots, i_{s}\right\} \in I$ is ordered, then
$$
c_{i_{1}, \ldots, i_{s}}=\frac{\operatorname{gcd}\left(u_{0}, \ldots, \hat{u}_{i_{1}}, \ldots, \hat{u}_{i_{s}}, \ldots, u_{n}\right)}{\prod_{J} c_{j_{1}, \ldots j_{t}}} .
$$
Similarly, we also define a set of $2^{s}$ real numbers by
$$
k_{\emptyset}=\epsilon_{n+1},
$$
and
$$
k_{i_{1}, \ldots, i_{s}}=\epsilon_{n-s+1} \sum_{I}(-1)^{s-t} \frac{u_{j_{1}} \cdots u_{j_{t}}}{v_{j_{1}} \cdots v_{j_{t}} \operatorname{lcm}\left(u_{j_{1}}, \ldots, u_{j_{t}}\right)}
$$
where
$$
\epsilon_{n-s+1}= \begin{cases}0 & \text { if } n-s+1 \text { is even } \\ 1 & \text { if } n-s+1 \text { is odd }\end{cases}
$$
respectively. Finally, for any $j$ such that $1 \leq j \leq r=\left\lfloor\max \left\{k_{i_{1}, \ldots, i_{s}}\right\}\right\rfloor$, where $\lfloor x\rfloor$ is the greatest integer less than or equal to $x$, we set
$
d_{j}=\prod_{k_{i_{1}, \ldots, i_{s}} \geq j} c_{i_{1}, \ldots, i_{s}}.$
Then 

\begin{equation}
H_{n-1}\left(L_{f}, \mathbb{Z}\right)_{\text {tor }}=\mathbb{Z} / d_{1} \oplus \cdots \oplus \mathbb{Z} / d_{r} .\end{equation}

\end{conjecture}

This conjecture was known to hold in certain special cases \cite{OR}. In a rather recent paper, Hertling and Mase \cite{HM}, extend the list of possible cases and showed that this conjecture is true  for the following cases:

\begin{enumerate}

\item {\it Chain type singularity:} that is,   a quasihomogeneous singularity of the form
$$
f=f\left(x_{1}, \ldots, x_{n}\right)=x_{1}^{a_{1}+1}+\sum_{i=2}^{n} x_{i-1} x_{i}^{a_{i}}
$$
for some $n \in \mathbb{N}$ and some $a_{1}, \ldots, a_{n} \in \mathbb{N}$. 
\item {\it Cycle type singularity:} a quasihomogeneous singularity of the form
$$
f=f\left(x_{1}, \ldots, x_{n}\right)=\sum_{i=1}^{n-1} x_{i}^{a_{i}} x_{i+1}+x_{n}^{a_{n}} x_{1}
$$
for some $n \in \mathbb{Z}_{\geq 2}$ and some $a_{1}, \ldots, a_{n} \in \mathbb{N}$ which satisfy for even $n$ neither $a_{j}=1$ for all even $j$ nor $a_{j}=1$ for all odd $j.$
\item {\it Thom-Sebastiani iterated sums of singularties of chain type or cyclic type.}  Recall, for  $f$ and $g$ be singularities, the Thom-Sebastiani sum is given by  $$f+g=f\left(x_{1}, \ldots, x_{n_{f}}\right)+g\left(x_{n_{f}+1}, \ldots, x_{n_{f}+n_{g}}\right).$$ Any iterated Thom-Sebastiani sum of chain type singularities and cycle type singularities are also called {\it invertible polynomials}. 
\item Although  {\it Brieskorn-Pham singularities}, or {\it BP singularities} $$f=f\left(x_{1}, \ldots, x_{n}\right)=\sum_{i=1}^{n} x_{i}^{a_{i}}$$ for some $n \in \mathbb{N}$ and some $a_{1}, \ldots, a_{n} \in \mathbb{Z}_{\geq 2}$ are special cases of chain type singularities we prefer to label them as an  independent type of singularity.
\end{enumerate}
In order to use Orlik's conjecture, we need to find from the two lists of Kähler-Einstein oribifolds aforementioned, elements with weights $\mathbf{w}=\left(w_{0}, w_{1}, w_{2}, w_{3}, w_{4}\right)$ and degree $d=\left(w_{0}+\cdots+w_{4}\right)-1$  that can be represented by BP singularities, chain type singularities, cycle type singularities or Thom Sebastiani sum of these types of singularities.
Thus, given a weight vector $\mathbf{w}=\left(w_{0}, w_{1}, w_{2}, w_{3}, w_{4}\right)$ one needs to determine if there exist exponents $a_{i}$'s  that  verify certain arithmetic condition. We include these conditions in the table below.

Most of the examples that we describe in this article have large weights and degrees, so in order to avoid monotonous  calculations,   codes were implemented in Matlab (see $(d)$ in Appendix) to determine  for a given  weight vector $\mathbf{w}=\left(w_{0}, w_{1}, w_{2}, w_{3}, w_{4}\right)$ the exponents $a_{i}$'s, such that the singularity can be written as a chain type, cycle type or Thom-Sebastiani sum of them. We also wrote a  code to computes the Betti numbers and the numbers $d_{i}$ which generate the torsion in $H_{3}\left(L_{f}, \mathbb{Z}\right)$.

\medskip

{\small
\begin{tabular}{| c | m{5cm} | m{6.4cm} |   } \hline
Type & \quad\quad\quad\quad\quad  Polynomial & \quad\quad \quad\quad \quad\quad \quad\quad Condition  \\ \hline \hline 

\tiny{BP}  & $z_{0}^{a_{0}}+z_{1}^{a_{1}}+z_{2}^{a_{2}}+z_{3}^{a_{3}}+z_{4}^{a_{4}}$ & $d=a_{0}w_{0}=a_{1}w_{1}=a_{2}w_{2}=a_{3}w_{3}=a_{4}w_{4}$   \\ \hline

\tiny{Chain}  & $z_{0}^{a_{0}}+z_{0}z_{1}^{a_{1}}+z_{1}z_{2}^{a_{2}}+z_{2}z_{3}^{a_{3}}+z_{3}z_{4}^{a_{4}}$ & $d=a_{0}w_{0}=w_{0}+a_{1}w_{1}=w_{1}+a_{2}w_{2}=w_{2}+ a_{3}w_{3}=w_{3}+a_{4}w_{4}$\\   \hline

\tiny{Cycle}  & $z_{4}z_{0}^{a_{0}}+z_{0}z_{1}^{a_{1}}+z_{1}z_{2}^{a_{2}}+z_{2}z_{3}^{a_{3}}+z_{3}z_{4}^{a_{4}}$ & $d=w_{4}+a_{0}w_{0}=w_{0}+a_{1}w_{1}=w_{1}+a_{2}w_{2}=w_{2}+ a_{3}w_{3}=w_{3}+a_{4}w_{4}$    \\\hline

 & $z_{0}^{a_{0}}+z_{1}^{a_{1}}+z_{1}z_{2}^{a_{2}}+z_{2}z_{3}^{a_{3}}+z_{3}z_{4}^{a_{4}}$ & $d=a_{0}w_{0}=a_{1}w_{1}=w_{1}+a_{2}w_{2}=w_{2}+ a_{3}w_{3}=w_{3}+a_{4}w_{4}$  \\ 
 
\tiny{BP + Chain}  & $z_{0}^{a_{0}}+z_{1}^{a_{1}}+z_{2}^{a_{2}}+z_{2}z_{3}^{a_{3}}+z_{3}z_{4}^{a_{4}}$ & $d=a_{0}w_{0}=a_{1}w_{1}=a_{2}w_{2}=w_{2}+ a_{3}w_{3}=w_{3}+a_{4}w_{4}$\\

  & $z_{0}^{a_{0}}+z_{1}^{a_{1}}+z_{2}^{a_{2}}+z_{3}^{a_{3}}+z_{3}z_{4}^{a_{4}}$ & $d=a_{0}w_{0}=a_{1}w_{1}=a_{2}w_{2}= a_{3}w_{3}=w_{3}+a_{4}w_{4}$  \\ \hline

 & $z_{0}^{a_{0}}+z_{4}z_{1}^{a_{1}}+z_{1}z_{2}^{a_{2}}+z_{2}z_{3}^{a_{3}}+z_{3}z_{4}^{a_{4}}$ & $d=a_{0}w_{0}=w_{4}+a_{1}w_{1}=w_{1}+a_{2}w_{2}=w_{2}+ a_{3}w_{3}=w_{3}+a_{4}w_{4}$  \\ 

\tiny{BP + Cycle}  & $z_{0}^{a_{0}}+z_{1}^{a_{1}}+z_{4}z_{2}^{a_{2}}+z_{2}z_{3}^{a_{3}}+z_{3}z_{4}^{a_{4}}$ & $d=a_{0}w_{0}=a_{1}w_{1}=w_{4}+a_{2}w_{2}=w_{2}+ a_{3}w_{3}=w_{3}+a_{4}w_{4}$\\

  & $z_{0}^{a_{0}}+z_{1}^{a_{1}}+z_{2}^{a_{2}}+z_{4}z_{3}^{a_{3}}+z_{3}z_{4}^{a_{4}}$ & $d=a_{0}w_{0}=a_{1}w_{1}=a_{2}w_{2}= w_{4}+a_{3}w_{3}=w_{3}+a_{4}w_{4}$  \\ \hline

  \tiny{Chain + Chain} & $z_{0}^{a_{0}}+z_{0}z_{1}^{a_{1}}+z_{2}^{a_{2}}+z_{2}z_{3}^{a_{3}}+z_{3}z_{4}^{a_{4}}$ & $d=a_{0}w_{0}=w_{0}+a_{1}w_{1}=a_{2}w_{2}= w_{2}+a_{3}w_{3}=w_{3}+a_{4}w_{4}$  \\ \hline
  
 \tiny{Chain + Cycle} & $z_{0}^{a_{0}}+z_{0}z_{1}^{a_{1}}+z_{4}z_{2}^{a_{2}}+z_{2}z_{3}^{a_{3}}+z_{3}z_{4}^{a_{4}}$ & $d=a_{0}w_{0}=w_{0}+a_{1}w_{1}=w_{4}+a_{2}w_{2}= w_{2}+a_{3}w_{3}=w_{3}+a_{4}w_{4}$ \\ 
 & $z_{0}^{a_{0}}+z_{0}z_{1}^{a_{1}}+z_{1}z_{2}^{a_{2}}+z_{4}z_{3}^{a_{3}}+z_{3}z_{4}^{a_{4}}$ & $d=a_{0}w_{0}=w_{0}+a_{1}w_{1}=w_{1}+a_{2}w_{2}= w_{4}+a_{3}w_{3}=w_{3}+a_{4}w_{4}$ \\ \hline
 
    \tiny{Cycle + Cycle} & $z_{1}z_{0}^{a_{0}}+z_{0}z_{1}^{a_{1}}+z_{4}z_{2}^{a_{2}}+z_{2}z_{3}^{a_{3}}+z_{3}z_{4}^{a_{4}}$ & $d=w_{1}+a_{0}w_{0}=w_{0}+a_{1}w_{1}=w_{4}+a_{2}w_{2}= w_{2}+a_{3}w_{3}=w_{3}+a_{4}w_{4}$ \\ 
 \hline
     
     \tiny{BP + Chain + Chain} & $z_{0}^{a_{0}}+z_{1}^{a_{1}}+z_{1}z_{2}^{a_{2}}+z_{3}^{a_{3}}+z_{3}z_{4}^{a_{4}}$ & $d=a_{0}w_{0}=a_{1}w_{1}=w_{1}+a_{2}w_{2}= a_{3}w_{3}=w_{4}+w_{3}+a_{4}w_{4}$\\ 
     \hline
     
     \tiny{BP + Chain + Cycle} & $z_{0}^{a_{0}}+z_{1}^{a_{1}}+z_{1}z_{2}^{a_{2}}+z_{4}z_{3}^{a_{3}}+z_{3}z_{4}^{a_{4}}$ & $d=a_{0}w_{0}=a_{1}w_{1}=w_{1}+a_{2}w_{2}= w_{4}+a_{3}w_{3}=w_{4}+w_{3}+a_{4}w_{4}$ \\ \hline
     \tiny{BP + Cycle + Cycle} & $z_{0}^{a_{0}}+z_{2}z_{1}^{a_{1}}+z_{1}z_{2}^{a_{2}}+z_{4}z_{3}^{a_{3}}+z_{3}z_{4}^{a_{4}}$ & $d=a_{0}w_{0}=w_{2}+a_{1}w_{1}=w_{1}+a_{2}w_{2}= w_{4}+a_{3}w_{3}=w_{4}+w_{3}+a_{4}w_{4}$ \\ \hline
\end{tabular}
}
\medskip

Next, we discuss an interesting result on the topology of the link if 
the degree $d$ and the weight vector ${\bf w}$ are such that $\operatorname{gcd}(d,w_{i})=1$ for all $i.$ Several elements on the tables we present in Section 3, satisfy this condition.  Notice that, if the type of singularities is restricted to the ten cases described on the  table given above, then $\operatorname{gcd}(d,w_{i})=1$  for all $i$ forces the singularity to be  of cycle type. We restrict to dimension 7, but this remark can be easily generalized for any dimension.

\begin{lemma} 
Consider a 7-manifold $M$ arising as a link of a hypersurface singularity with degree $d$ and such that $gcd(d,w_{i})=1$ for all $i=0,\dots 4$, then 
$\mu + 1 = d(b_{3}+1)$ and $H_{3}(M,\ZZ)_{tor}=\ZZ_{d}.$ In particular, for hypersurface singularities that determine a rational homology sphere with $\operatorname{gcd}\left(d, w_{i}\right)=1$ for all $i$ we obtain exactly $d-1$ vanishing cycles.
\end{lemma}

\begin{proof}
Indeed, from Equation $(6)$ we have 
$$b_{3} = \sum(-1)^{5-s}\left[\dfrac{u_{i_{1}}\dots u_{i_{s}}}{v_{i_{1}}\dots v_{i_{s}}lcm(u_{i_{1}}\dots u_{i_{s}})}\right],$$
where the sum is over all subsets $\{ i_{1}\dots i_{s}\}\subset\{0,1,2,3,4\}$ with $ i_{1}<\dots< i_{s}$  and
$$u_{i}=\dfrac{d}{gcd(d,w_{i})}, v_{i}=\dfrac{w_{i}}{gcd(d,w_{i})}, \forall i=0,\dots,4.$$
Since $gcd(d,w_{i})=1$ we have  $u_{i}=d$ and $v_{i}=w_{i}$. Then the above formula can be simplified as follows 
$$b_{3} = -1 + \sum \dfrac{1}{w_{i_{1}}}-\sum \dfrac{d}{w_{i_{1}}w_{i_{2}}}+\sum \dfrac{d^{2}}{w_{i_{1}}w_{i_{2}}w_{i_{3}}}-\sum \dfrac{d^{3}}{w_{i_{1}}w_{i_{2}}w_{i_{3}}w_{i_{4}}}+\dfrac{d^{4}}{w_{0}w_{1}w_{2}w_{3}w_{4}}.$$
Writing 
$
    S_{1}=\sum w_{i_{1}}, \ \ S_{2}  =\sum w_{i_{1}}w_{i_{2}}, \ \ S_{3}  =\sum w_{i_{1}}w_{i_{2}}w_{i_{3}}, \ \ S_{4}=\sum w_{i_{1}}w_{i_{2}}w_{i_{3}}w_{i_{4}},
$
we obtain
\begin{equation}
    b_{3} = -1 + \dfrac{S_{4}}{w_{0}\dots w_{4}}-\dfrac{S_{3}d}{w_{0}\dots w_{4}}+\dfrac{S_{2}d^{2}}{w_{0}\dots w_{4}}-\dfrac{S_{1}d^{3}}{w_{0}\dots w_{4}}+\dfrac{d^{4}}{w_{0}\dots w_{4}}
 \end{equation}
    
On the other hand, the Milnor number is given by 
$\mu = \left(\dfrac{d}{w_{0}}-1\right)\dots \left(\dfrac{d}{w_{4}}-1\right).$
Then 
$$\mu = \dfrac{d^{5}-S_{1}d^{4}+S_{2}d^{3}-S_{3}d^{2}+S_{4}d-w_{0}\dots w_{4}}{w_{0}\dots w_{4}}.$$
From (8), we obtain 
$\mu + 1 = d(b_{3}+1).$
For the  second claim, one notices that $gcd(d,w_{i})=1$ implies  $u_{i}=d$ and $v_{i}=w_{i}$, for all $i$. Thus $c_{\emptyset}=gcd(d,\dots, d) = d \ \ \mbox{and} \ \ \ c_{i_{1}\dots i_{s}}=1.$
 Clearly $k_{\emptyset}=1$. Then, for $j=1$, we get $d_{1}=d$, while in other cases, $d_{j}=1$. Therefore, from Equation $(7)$ one concludes that  if $\operatorname{gcd}(d,w_{i})=1$ for all $i=0,\dots n$ then  
   $H_{3}(M,\ZZ)_{tor}=\ZZ_{d}.$
\end{proof}

\section{Results}

In this section we present three tables of new examples of two classes of 7-manifolds: rational homology 7-spheres, and 7-manifolds homeomorphic to connected sums of $S^3\times S^4$, all  admitting Sasaki-Einstein metrics. These new examples are extracted from two lists of Kähler-Einstein orbifolds, the first one given by  Johnson and Kollár  in \cite{JK} (see $(a)$ in Appendix for link to the list) and the other one given by Cheltsov in \cite{Ch}. In these three tables we list the weights, the quasihomogenous polynomial generating the link, the type of singularity,  the degree, the Milnor number and finally the third homology group.  But, first, we discuss the diffeomorphism type of certain homotopy spheres admitting positive Ricci curvature.

\subsection{Positive Ricci curvature on homotopy spheres}

The existence of positive Sasakian metrics on homotopy spheres was study in detail by Boyer, Galicki and collaborators, where the authors exhibited inequivalent  families of homotopy spheres admitting this type of metrics based on their list of 184  rational homology 7-spheres admitting Sasakian-Einstein metrics. We apply  the methods developed in \cite{BGN2} and \cite{BGN3} to the list of the  new 52  rational homology 7-spheres presented in this paper. Let us briefly review some basic facts  on links viewed as branched covers. The main reference here is \cite{BGN3}.

Let $f=f\left(z_1, \ldots, z_m\right)$ be a quasi-smooth weighted homogeneous polynomial of degree $d_f$ in $m$ complex variables, and let $L_f$ denote its link. Let $\mathbf{w}_f=\left(w_1, \ldots, w_m\right)$ be the corresponding weight vector. We consider branched covers constructed as the link $L_g$ of the polynomial
$$
g=z_0^p+f\left(z_1, \ldots, z_m\right), 
$$
for $p>1.$ 
Then $L_g$ is a $p$-fold branched cover of $S^{2 m+1}$ branched over the link $L_f$. The degree of $L_g$ is $d_g=\operatorname{lcm}\left(p, d_f\right)$, and the weight vector is $\mathbf{w}_g=\left(\frac{d_f}{\operatorname{gcd}\left(p, d_f\right)}, \frac{p}{\operatorname{gcd}\left(p, d_f\right)} \mathbf{w}_f\right)$.

We have 
\begin{equation}
|{\bf w}_g|-d_g=\frac{d_f+ p(|{\bf w}_f|-d_f)}{\operatorname{gcd}\left(p, d_f\right)}.
\end{equation}
Thus, from Equation (1) and Theorem 2.1 it follows that $L_g$ admits positive Ricci curvature for all $p>0$ if $L_f$ admits a Sasaki-Einstein metric.

Furthermore, in \cite{BGN3} the authors showed that  the link $L_g$  is a homotopy sphere if and only if  $L_f $ is a rational homology sphere.  Let us recall the argument, since it will be useful for our purpose.  Here we assume  that $\operatorname{gcd}\left(p, d_f\right)=1.$

For the divisor $\operatorname{div} \Delta_g$ we have the equalities
$$
\begin{aligned}
\operatorname{div} \Delta_g & =\left(\Lambda_p-1\right) \operatorname{div} \Delta_f=\left(\Lambda_p-1\right)\left((-1)^n+\sum a_j \Lambda_j\right) \\
& =\sum_j \operatorname{gcd}(p, j) a_j \Lambda_{\operatorname{lcm}(p, j)}-\sum_j a_j \Lambda_j+(-1)^n \Lambda_p+(-1)^{n+1} .
\end{aligned}
$$
Since the $j$ 's run through all the least common multiples of the set $\left\{u_1, \ldots, u_n\right\}$ and $\operatorname{gcd}\left(p, u_i\right)=1$ for all $i$, we note  that for all $j, \operatorname{gcd}(p, j)=1$. Thus
$$
b_{n-1}\left(L_g\right)=\sum_j a_j-\sum_j a_j+(-1)^n+(-1)^{n+1}=0.
$$

So $L_g$ is a rational homology sphere. Now, the  computation of the Alexander polynomial for $L_g$ leads to:
\begin{align}
\Delta_g(t) & =(t-1)^{(-1)^{n+1}}\left(t^p-1\right)^{(-1)^n} \prod_j\left(t^{p j}-1\right)^{a_j}\left(t^j-1\right)^{-a_j} \nonumber\\
& =\left(t^{p-1}+\cdots+t+1\right)^{(-1)^n} \prod_j\left(\frac{t^{p j-1}+\cdots+t+1}{t^{j-1}+\cdots+t+1}\right)^{a_j} \nonumber\\ 
&=\left(t^{p-1}+\cdots+t+1\right)^{(-1)^n} \prod_j\left(t^{(p-1) j}+\cdots+t^j+1\right)^{a_j}. \label{eq:test}
\end{align}
This gives
$$
\Delta_g(1)=p^{\Sigma_j a_j+(-1)^n}.$$ 
Thus from Equation (5) it follows that $L_g$ is a homotopy sphere if and only if $L_f$ is a rational homology sphere.

In \cite{BGN3}, the authors made use of  Equation (10) and apply  Theorem 2.2  to determine the diffeomorphism type for $L_g$   for  7-dimensional links $L_f$ coming from their list of rational homology 7-spheres \cite{BGN2}. If  $L_f$ has {\it odd degree} they proved.  

\begin{theorem}[Boyer et al. \cite{BGN2}] Let $L_g$ be the link of a $p$-branched cover of $S^{2n+1}$ branched over a link 
 $L_f$   in their list of 184  rational homology spheres given in \cite{BGN2}. Suppose that degree $d_f$ of $L_f$ is odd and that  $\operatorname{gcd}(p, d_f)=1.$ Then, for  $p$  odd, $L_g$ is diffeomorphic to the standard 9-sphere $S^9.$ For  $p$  even,  $L_g$ is diffeomorphic to $S^9$ if $|H_3(L_f, \mathbb Z)|\equiv \pm 1 (\bmod 8)$ and $L_g$ is diffeomorphic to the Kervaire exotic sphere $\Sigma_9$ if $|H_3(L_f, \mathbb Z)|\equiv \pm 3 (\bmod 8).$
\end{theorem}

The next theorem is an update of Theorems 7.6 and Theorems 7.7 in \cite{BGN3} and its  proof is similar to the one given by Boyer et al.

\begin{theorem}  Let us assume for  $\operatorname{gcd}(p, d_f)=1$, then 
 for each  member of the 52 rational homology spheres in our list given in Table 1, $L_g$ is a homotopy sphere admitting positive Ricci curvature. 
If $L_f$ is taken to be one of the 49 rational homology spheres with odd degree listed in Table 1, then $L_g$ admits Sasakian metrics with positive Ricci curvature and these are  diffeomorphic to the standard sphere $S^9.$ For the remaining three members with even degree  in Table 1,  we have 
\begin{enumerate}
\item For $L_f$ with ${\bf w}=(118, 118, 185, 135, 35)$ with  degree $d_f=590,$ $L_g$ diffeomorphic to the standard $S^9$ if $p^4\equiv \pm 1 (\bmod 8)$ and diffeomorphic to the exotic Kervaire 9-sphere $\Sigma_9$ if $p^4\equiv \pm 3 (\bmod 8).$

\item For $L_f$ with ${\bf w}=(64, 512, 475, 375, 175)$ with  degree $d_f=1600,$ $L_g$ diffeomorphic to the standard $S^9$ if $p^3\equiv \pm 1 (\bmod 8)$ and diffeomorphic to the exotic Kervaire 9-sphere $\Sigma_9$ if $p^3\equiv \pm 3 (\bmod 8).$

\item For $L_f$ with ${\bf w}=(3532, 7064, 5355, 115, 1595)$ with degree $d_f=17660,$ $L_g$ diffeomorphic to the standard $S^9$ if $p^2\equiv \pm 1 (\bmod 8)$ and diffeomorphic to the exotic Kervaire 9-sphere $\Sigma_9$ if $p^2\equiv \pm 3 (\bmod 8).$
\end{enumerate}
\end{theorem}
\begin{proof}
Our list of rational homology 7-spheres admitting Sasaki-Einstein metrics contains 49 elements with odd degree and all of these have order $\left|H_{3}\right| \equiv 1 (\bmod 8),$ thus, the first part of the theorem follows from Theorem 3.1. 
For the three members with even degree in Table 1, the weight vectors  can be written as  
$\textbf{w}=(w_{0},w_{1},w_{2},w_{3},w_{4})=(m_{2}v_{0},m_{2}v_{1},m_{2}v_{2},m_{3}v_{3},m_{3}v_{4})$ 
where $\gcd(m_{2},m_{3})=1$ and $m_{2}m_{3}=d$ where  $m_2$ is odd and $m_3$ is even. For these one obtains 
\begin{equation}
\operatorname{div}\Delta_f=n({\bf w})\Lambda_d+\Lambda_{m_3}-n({\bf w})\Lambda_{m_2}-1,
\end{equation}
 where $n({\bf w})$ is a positive integer.  It follows from Equation (4) that $\sum_{j\, even}a_j=n({\bf w})+1$. 

On the other hand, bearing in mind Theorem 2.2, we pay attention to $\Delta_g(-1):$  since $p$ is odd and $d$ even,  from Equation (10) we have 

$$\Delta_g(-1)=\prod_{j \text { even }} p^{a_j}=p^{\Sigma_{j \text { even }}^{a_j}}$$

Since $L_f$ is a rational homology sphere $H_3(L_f,\mathbb Z)_{tor}=\Delta(1),$ so in order to compute $\Delta(1)$ we rewrite the equation (11) in terms of the Alexander polynomial:  
$$
\begin{aligned}
\Delta(t) &=& (t^d-1)^{n({\bf w})}(t^{m_3}-1)(t^{m_2}-1)^{-n({\bf w})}(t-1)^{-1}\\
&=& (t^{d-1}+\ldots 1)^{n({\bf w})}(t^{m_3-1}+\ldots 1)(t^{m_2-1}+\ldots 1)^{-n({\bf w})}, 
\end{aligned}
$$
Thus $\Delta(1)=dm_3m_2^{-n({\bf w})}=m_3^{n({\bf w})+1}.$ The result follows by checking the torsion of $H_3(L_f, \mathbb Z)$ at Table 1 below.

\end{proof}

\subsection{\textit{Rational homology 7-spheres admitting Sasaki-Einstein structures from Johnson-Koll\'ar list}}

In  \cite{JK} Johnson and Kollár gave a list of 4442 of well-formed quasismooth $\mathbb{Q}$-Fano 3-folds of index one anti-canonically embedded in $\mathbb{C P}^{4}(\mathbf{w}).$ This list contains  1936 3-folds that admit a Kähler-Einstein metric. Theorem 2.1 implies that the corresponding link  admits Sasaki-Einstein metrics.  These links were studied in \cite{BGN2} and they  found 184 rational homology 7-spheres.  In the following table, we present 52 new examples of 2-connected rational homology spheres admitting Sasaki-Einstein metrics. 
\medskip

\begin{center}
\noindent{\bf Table 1: Rational homology 7-spheres admitting S-E structures} 
\end{center}
\medskip
{\Small{
\begin{longtable}{| c | c | c | c | c | c | } \hline

${\bf{w}}=(w_0,w_1,w_2, w_3, w_4)$ & Polynomial & Type &   $d$   &  $\mu$  & $H_{3}(M,\ZZ)$   \\ \hline \hline \endfirsthead

\hline
 ${\bf{w}}=(w_0,w_1,w_2, w_3, w_4)$ & Polynomial & Type & $d$   &  $\mu$  & $H_{3}(M,\ZZ)$  \\ \hline \hline \endhead

(13,143,775,620,465)  & $z_{0}^{155}+z_{0}z_{1}^{14}+z_{4}z_{2}^{2}+z_{2}z_{3}^{2}+z_{3}z_{4}^{3}$ & \tiny{Chain + Cycle} & 2015 & 24192 & $(\ZZ_{13})^{14}$   \\  \hline

(77,77,333,180,27)  & $z_{0}^{9}+z_{1}^{9}+z_{4}z_{2}^{2}+z_{2}z_{3}^{2}+z_{3}z_{4}^{19}$ & \tiny{BP + Cycle} & 693 & 4864 & $(\ZZ_{77})^{8}$   \\  \hline

(67,67,161,28,147)  & $z_{0}^{7}+z_{1}^{7}+z_{4}z_{2}^{2}+z_{2}z_{3}^{11}+z_{3}z_{4}^{3}$ & \tiny{BP + Cycle} & 469 & 2376 & $(\ZZ_{67})^{6}$   \\  \hline

(29,667,1807,1112,417)  & $z_{0}^{139}+z_{0}z_{1}^{6}+z_{4}z_{2}^{2}+z_{2}z_{3}^{2}+z_{3}z_{4}^{7}$ & \tiny{Chain + Cycle} & 4031 & 19488 & $(\ZZ_{29})^{6}$   \\  \hline

(493,34,1841,1315,789)  & $z_{1}z_{0}^{9}+z_{0}z_{1}^{117}+z_{4}z_{2}^{2}+z_{2}z_{3}^{2}+z_{3}z_{4}^{4}$ & \tiny{Cycle + Cycle} & 4471 & 16848 & $(\ZZ_{17})^{5}$   \\  \hline

(67,67,217,84,35)  & $z_{0}^{7}+z_{1}^{7}+z_{4}z_{2}^{2}+z_{2}z_{3}^{3}+z_{3}z_{4}^{11}$ & \tiny{BP + Cycle} & 469 & 2376 & $(\ZZ_{67})^{6}$   \\  \hline

(118,118,185,135,35)  & $z_{0}^{5}+z_{1}^{5}+z_{4}z_{2}^{3}+z_{2}z_{3}^{3}+z_{3}z_{4}^{13}$ & \tiny{BP + Cycle} & 590 & 1872 & $(\ZZ_{118})^{4}$   \\  \hline

(373,373,780,35,305)  & $z_{0}^{5}+z_{1}^{5}+z_{4}z_{2}^{2}+z_{2}z_{3}^{31}+z_{3}z_{4}^{6}$ & \tiny{BP + Cycle} & 1865 & 5952 & $(\ZZ_{373})^{4}$   \\  \hline

(113,226,715,377,39)  & $z_{0}^{13}+z_{0}z_{1}^{6}+z_{4}z_{2}^{2}+z_{2}z_{3}^{2}+z_{3}z_{4}^{28}$ & \tiny{Chain + Cycle} & 1469 & 7392 & $(\ZZ_{113})^{6}$   \\  \hline

(253,253,545,40,175)  & $z_{0}^{5}+z_{1}^{5}+z_{4}z_{2}^{2}+z_{2}z_{3}^{18}+z_{3}z_{4}^{7}$ & \tiny{BP + Cycle} & 1265 & 4032 & $(\ZZ_{253})^{4}$   \\  \hline

(43,1333,1875,500,1625)  & $z_{0}^{125}+z_{0}z_{1}^{4}+z_{4}z_{2}^{2}+z_{2}z_{3}^{7}+z_{3}z_{4}^{3}$ & \tiny{Chain + Cycle} & 5375 & 15792 & $(\ZZ_{43})^{4}$   \\  \hline

(43,1333,2375,1000,625)  & $z_{0}^{125}+z_{0}z_{1}^{4}+z_{4}z_{2}^{2}+z_{2}z_{3}^{3}+z_{3}z_{4}^{7}$ &  \tiny{Chain + Cycle} & 5375 & 15792 & $(\ZZ_{43})^{4}$   \\  \hline

(73,73,95,45,80)  & $z_{0}^{5}+z_{1}^{5}+z_{4}z_{2}^{3}+z_{2}z_{3}^{6}+z_{3}z_{4}^{4}$ & \tiny{BP + Cycle} & 365 & 1152 & $(\ZZ_{73})^{4}$   \\  \hline

(185,740,1911,987,63)  & $z_{0}^{21}+z_{0}z_{1}^{5}+z_{4}z_{2}^{2}+z_{2}z_{3}^{2}+z_{3}z_{4}^{46}$ & \tiny{Chain + Cycle} & 3885 & 15640 & $(\ZZ_{185})^{5}$   \\  \hline

(929,1858,2849,63,805)  & $z_{0}^{7}+z_{0}z_{1}^{3}+z_{4}z_{2}^{2}+z_{2}z_{3}^{58}+z_{3}z_{4}^{8}$ & \tiny{Chain + Cycle} & 6503 & 13920 & $(\ZZ_{929})^{3}$   \\  \hline

(64,512,475,375,175)  & $z_{0}^{25}+z_{0}z_{1}^{3}+z_{4}z_{2}^{3}+z_{2}z_{3}^{3}+z_{3}z_{4}^{7}$ & \tiny{Chain + Cycle} & 1600 & 3213 & $(\ZZ_{64})^{3}$   \\  \hline

(253,253,600,95,65)  & $z_{0}^{5}+z_{1}^{5}+z_{4}z_{2}^{2}+z_{2}z_{3}^{7}+z_{3}z_{4}^{18}$ & \tiny{BP + Cycle} & 1265 & 4032 & $(\ZZ_{253})^{4}$   \\  \hline

(127,381,793,286,65)  & $z_{0}^{13}+z_{0}z_{1}^{4}+z_{4}z_{2}^{2}+z_{2}z_{3}^{3}+z_{3}z_{4}^{21}$ & \tiny{Chain + Cycle} & 1651 & 5040 & $(\ZZ_{127})^{4}$   \\  \hline

(65,650,1581,867,153)  & $z_{0}^{51}+z_{0}z_{1}^{5}+z_{4}z_{2}^{2}+z_{2}z_{3}^{2}+z_{3}z_{4}^{16}$ & \tiny{Chain + Cycle} & 3315 & 13120 & $(\ZZ_{65})^{5}$   \\  \hline

(231,66,481,185,259)  & $z_{1}z_{0}^{5}+z_{0}z_{1}^{15}+z_{4}z_{2}^{2}+z_{2}z_{3}^{4}+z_{3}z_{4}^{4}$ & \tiny{Cycle + Cycle} & 1221 & 2400 & $(\ZZ_{33})^{3}$   \\  \hline

(1003,68,3745,2675,1605)  & $z_{1}z_{0}^{9}+z_{0}z_{1}^{119}+z_{4}z_{2}^{2}+z_{2}z_{3}^{2}+z_{3}z_{4}^{4}$ & \tiny{Cycle + Cycle} & 9095 & 17136 & $(\ZZ_{17})^{3}$   \\  \hline

(73,584,1435,779,123)  & $z_{0}^{41}+z_{0}z_{1}^{5}+z_{4}z_{2}^{2}+z_{2}z_{3}^{2}+z_{3}z_{4}^{18}$ & \tiny{Chain + Cycle} & 2993 & 11880 & $(\ZZ_{73})^{5}$   \\  \hline

(481,962,1519,77,329)  & $z_{0}^{7}+z_{0}z_{1}^{3}+z_{4}z_{2}^{2}+z_{2}z_{3}^{24}+z_{3}z_{4}^{10}$ & \tiny{Chain + Cycle} & 3367 & 7200 & $(\ZZ_{481})^{3}$   \\  \hline

(657,3942,4693,95,3097)  & $z_{0}^{19}+z_{0}z_{1}^{3}+z_{4}z_{2}^{2}+z_{2}z_{3}^{82}+z_{3}z_{4}^{4}$ & \tiny{Chain + Cycle} & 12483 & 25584 & $(\ZZ_{657})^{3}$   \\  \hline

(2628,1971,4693,95,3097)  & $z_{1}z_{0}^{4}+z_{0}z_{1}^{5}+z_{4}z_{2}^{2}+z_{2}z_{3}^{82}+z_{3}z_{4}^{4}$ & \tiny{Cycle + Cycle} & 12483 & 13120 & $(\ZZ_{657})^{2}$   \\  \hline

(3773,98,8901,5031,1161)  & $z_{1}z_{0}^{5}+z_{0}z_{1}^{155}+z_{4}z_{2}^{2}+z_{2}z_{3}^{2}+z_{3}z_{4}^{12}$ & \tiny{Cycle + Cycle} & 18963 & 37200 & $(\ZZ_{49})^{3}$   \\  \hline

(2069,2069,1413,102,555)  & $z_{0}^{3}+z_{1}^{3}+z_{4}z_{2}^{4}+z_{2}z_{3}^{47}+z_{3}z_{4}^{11}$ & \tiny{BP + Cycle} & 6207 & 8272 & $(\ZZ_{2069})^{2}$   \\  \hline

(929,1858,3199,413,105)  & $z_{0}^{7}+z_{0}z_{1}^{3}+z_{4}z_{2}^{2}+z_{2}z_{3}^{8}+z_{3}z_{4}^{58}$ & \tiny{Chain + Cycle} & 6503 & 13920 & $(\ZZ_{929})^{3}$   \\  \hline

(3532,7064,5355,115,1595)  & $z_{0}^{5}+z_{0}z_{1}^{2}+z_{4}z_{2}^{3}+z_{2}z_{3}^{107}+z_{3}z_{4}^{11}$ & \tiny{Chain + Cycle}  & 17660 & 21186 & $(\ZZ_{3532})^{2}$   \\  \hline

(1505,6020,3357,2547,117)  & $z_{0}^{9}+z_{0}z_{1}^{2}+z_{4}z_{2}^{4}+z_{2}z_{3}^{4}+z_{3}z_{4}^{94}$ & \tiny{Chain + Cycle} & 13545 & 15040 & $(\ZZ_{1505})^{2}$   \\  \hline

(136,119,889,635,381)  & $z_{1}z_{0}^{15}+z_{0}z_{1}^{17}+z_{4}z_{2}^{2}+z_{2}z_{3}^{2}+z_{3}z_{4}^{4}$ & \tiny{Cycle + Cycle} & 2159 & 4080 & $(\ZZ_{17})^{3}$   \\  \hline

(1297,3891,2653,119,1120)  & $z_{0}^{7}+z_{0}z_{1}^{2}+z_{4}z_{2}^{3}+z_{2}z_{3}^{54}+z_{3}z_{4}^{8}$ & \tiny{Chain + Cycle} & 9079 & 10368 & $(\ZZ_{1297})^{2}$   \\  \hline

(6485,2594,9197,119,3655)  & $z_{1}z_{0}^{3}+z_{0}z_{1}^{6}+z_{4}z_{2}^{2}+z_{2}z_{3}^{108}+z_{3}z_{4}^{6}$ & \tiny{Cycle + Cycle} & 22049 & 23328 & $(\ZZ_{1297})^{2}$   \\  \hline

(1457,1457,1011,120,327)  & $z_{0}^{3}+z_{1}^{3}+z_{4}z_{2}^{4}+z_{2}z_{3}^{28}+z_{3}z_{4}^{13}$ & \tiny{BP + Cycle} & 4371 & 5824 & $(\ZZ_{1457})^{2}$   \\  \hline

(701,701,381,123,198)  & $z_{0}^{3}+z_{1}^{3}+z_{4}z_{2}^{5}+z_{2}z_{3}^{14}+z_{3}z_{4}^{10}$ & \tiny{BP + Cycle} & 2103 & 2800 & $(\ZZ_{701})^{2}$   \\  \hline

(2069,2069,1521,426,123)  & $z_{0}^{3}+z_{1}^{3}+z_{4}z_{2}^{4}+z_{2}z_{3}^{11}+z_{3}z_{4}^{47}$ & \tiny{BP + Cycle} & 6207 & 8272 & $(\ZZ_{2069})^{2}$   \\  \hline

(2149,921,3193,124,3131)  & $z_{1}z_{0}^{4}+z_{0}z_{1}^{8}+z_{4}z_{2}^{2}+z_{2}z_{3}^{51}+z_{3}z_{4}^{3}$ & \tiny{Cycle + Cycle} & 9517 & 9792 & $(\ZZ_{307})^{2}$   \\  \hline

(289,2312,2725,125,1775)  & $z_{0}^{25}+z_{0}z_{1}^{3}+z_{4}z_{2}^{2}+z_{2}z_{3}^{36}+z_{3}z_{4}^{4}$ &  \tiny{Chain + Cycle} & 7225 & 14688 & $(\ZZ_{289})^{3}$   \\  \hline

(129,3612,4165,425,2635)  & $z_{0}^{85}+z_{0}z_{1}^{3}+z_{4}z_{2}^{2}+z_{2}z_{3}^{16}+z_{3}z_{4}^{4}$ & \tiny{Chain + Cycle}& 10965 & 21888 & $(\ZZ_{129})^{3}$   \\  \hline

(129,3612,5185,1445,595)  & $z_{0}^{85}+z_{0}z_{1}^{3}+z_{4}z_{2}^{2}+z_{2}z_{3}^{4}+z_{3}z_{4}^{16}$ & \tiny{Chain + Cycle} & 10965 & 21888 & $(\ZZ_{129})^{3}$   \\  \hline

(4085,129,5745,1532,4979)  & $z_{1}z_{0}^{4}+z_{0}z_{1}^{96}+z_{4}z_{2}^{2}+z_{2}z_{3}^{7}+z_{3}z_{4}^{3}$ & \tiny{Cycle + Cycle} & 16469 & 16128 & $(\ZZ_{43})^{2}$   \\  \hline

(4085,129,7277,3064,1915)  & $z_{1}z_{0}^{4}+z_{0}z_{1}^{96}+z_{4}z_{2}^{2}+z_{2}z_{3}^{3}+z_{3}z_{4}^{7}$ & \tiny{Cycle + Cycle} & 16469 & 16128 & $(\ZZ_{43})^{2}$   \\  \hline

(481,962,1617,175,133)  & $z_{0}^{7}+z_{0}z_{1}^{3}+z_{4}z_{2}^{2}+z_{2}z_{3}^{10}+z_{3}z_{4}^{24}$ & \tiny{Chain + Cycle} & 3367 & 7200 & $(\ZZ_{481})^{3}$   \\  \hline
 
(657,3942,6175,1577,133)  & $z_{0}^{19}+z_{0}z_{1}^{3}+z_{4}z_{2}^{2}+z_{2}z_{3}^{4}+z_{3}z_{4}^{82}$ & \tiny{Chain + Cycle} & 12483 & 25584 & $(\ZZ_{657})^{3}$   \\  \hline 

(2628,1971,6175,1577,133)  & $z_{1}z_{0}^{4}+z_{0}z_{1}^{5}+z_{4}z_{2}^{2}+z_{2}z_{3}^{4}+z_{3}z_{4}^{82}$ & \tiny{Cycle + Cycle} & 12483 & 13120 & $(\ZZ_{657})^{2}$   \\  \hline

(1945,477,1321,148,1871)  & $z_{4}z_{0}^{2}+z_{0}z_{1}^{8}+z_{1}z_{2}^{4}+z_{2}z_{3}^{30}+z_{3}z_{4}^{3}$ & \tiny{Cycle} & 5761 & 5760 & $\ZZ_{5761}$   \\  \hline

(2387,1579,661,148,771)  & $z_{4}z_{0}^{2}+z_{0}z_{1}^{2}+z_{1}z_{2}^{6}+z_{2}z_{3}^{33}+z_{3}z_{4}^{7}$ & \tiny{Cycle} & 5545 & 5544 & $\ZZ_{5545}$   \\  \hline

(9142,3097,1917,4129,149)  & $z_{4}z_{0}^{2}+z_{0}z_{1}^{3}+z_{1}z_{2}^{8}+z_{2}z_{3}^{4}+z_{3}z_{4}^{96}$ & \tiny{Cycle} & 18433 & 18432 & $\ZZ_{18433}$   \\  \hline

(2323,1611,562,151,899)  & $z_{4}z_{0}^{2}+z_{0}z_{1}^{2}+z_{1}z_{2}^{7}+z_{2}z_{3}^{33}+z_{3}z_{4}^{6}$ & \tiny{Cycle} & 5545 & 5544 & $\ZZ_{5545}$   \\  \hline

(3073,712,2211,151,1199)  & $z_{4}z_{0}^{2}+z_{0}z_{1}^{6}+z_{1}z_{2}^{3}+z_{2}z_{3}^{34}+z_{3}z_{4}^{6}$ & \tiny{Cycle} & 7345 & 7344 & $\ZZ_{7345}$   \\  \hline

(1585,189,1105,292,1439)  & $z_{4}z_{0}^{2}+z_{0}z_{1}^{16}+z_{1}z_{2}^{4}+z_{2}z_{3}^{12}+z_{3}z_{4}^{3}$ & \tiny{Cycle} & 4609 & 4608 & $\ZZ_{4609}$   \\  \hline

(18277,6172,10207,1899,239)  & $z_{4}z_{0}^{2}+z_{0}z_{1}^{3}+z_{1}z_{2}^{3}+z_{2}z_{3}^{14}+z_{3}z_{4}^{146}$ & \tiny{Cycle} & 36793 & 36792 & $\ZZ_{36793}$   \\  \hline

\end{longtable}
}}

All the rational homology 7-spheres that we exhibit in Table 1 come from polynomials that are of cycle type or that have that type of singularity as part of its Thom-Sebastiani representation. The examples found in \cite{BGN2} also have that particular feature. Actually, in \cite{BGN2}, Lemma 3.3  provides examples of cycle type singularities,  while Lemma 3.10 provides examples of cycle+cycle type singularities.  Indeed,  if we can write the vector 
$\textbf{w}=(w_{0},w_{1},w_{2},w_{3},w_{4})=(m_{2}v_{0},m_{2}v_{1},m_{2}v_{2},m_{3}v_{3},m_{3}v_{4})$ 
where $\gcd(m_{2},m_{3})=1$ and $m_{2}m_{3}=d$, then Lemma 3.10 in \cite{BGN2} implies that
the  number $$\displaystyle{n(\textbf{w})=\dfrac{m_{2}}{v_{3}v_{4}}-\dfrac{1}{v_{3}}-\dfrac{1}{v_{4}}}$$ is a positive integer.
This is equivalent to $$m_{2}=v_{3}+v_{4}(1+n(\textbf{w})v_{3}).$$
Multiplying by $m_{3}$, we have 
\begin{equation}
d=w_{3}+w_{4}a_{4}
\end{equation}
where $a_{4}>1$ is integer. Analogously, we get 
$m_{2}=v_{4}+v_{3}(1+n(\textbf{w})v_{4}).$ 
Multiplying by $m_{2}$, we obtain
\begin{equation}
d=w_{4}+w_{3}a_{3}
\end{equation}
From (9) and (10) it follows that there exist a cycle block
$z_{4}z_{3}^{a_{3}}+z_{3}z_{4}^{a_{4}}$ that is a summand of some invertible polynomial associated to $\textbf{w}$.
We conjecture that all rational homology 7-spheres arising as links of  hypersurface singularities can be given by  polynomials that  must contain a cycle singularity term in its representation as a Thom-Sebastiani sum, at least when  $d=\sum_{i}^4w_i+1.$

\begin{remark} 
In \cite{BBG}, from links  $L_f'$ of quasi-smooth weighted hypersurfaces $f'(z_2\ldots z_n)$, the authors  consider the  quasi-smooth weighted hypersurfaces 
$$f(z_0, \ldots z_n)=z_0^2+z_1^2+f'$$ in $\mathbb C^{n+1}.$  It is not difficult to show that the corresponding link $L_f$ admits Sasakian structures with positive Ricci curvature. Additionally, if  the link $L_f'$ is a rational homology sphere,   from Sebastiani-Thom theorem \cite{ST}   it follows that the link $L_f$ is
a rational homology sphere. Thus, from each of the new  52 rational homology 7-spheres one can produce a rational homology 11-sphere admitting positive Ricci curvature. 
\end{remark}

Following the terminology given in  \cite{BGN2}, a rational homology 7-spheres with the same degree $d$, Milnor number $\mu$ and order of $H_{3}$ is a twin. We  detect ten {\it twins} in Table 1;
except the couple given by the weights  $(2323,1611,562,151,899)$ and $(2387,1579,661,148,771),$ both of them of cycle type, with 
$d=|H_3|=5545$ and $\mu=5544$, the rest of twins  have weight vectors with identical first two components $w_0, w_1.$ As mentioned in \cite{BGN2}, it is tempting to conjecture that twins are homeomorphic or even diffeomorphic links, however we are not able to determine this. In an upcoming article, we study  the behavior of  twins trough an operation associated to the Berglund-Hübsch transpose rule from BHK mirror symmetry \cite{BH}.

\subsection{\textit{Sasaki-Einstein  7-manifolds of the form $\#k(S^3\times S^4)$ from the Johnson-Koll\'ar's list and Cheltsov's list}}

In \cite{BG1} it is proven that a  $2$-connected oriented  7-manifold $M$ that bounds a  parallelizable 8-manifold with  $H_3(M, \mathbb Z)$  torsion free is completely determined up to diffeomorphism by the rank of $H_3(M, \mathbb Z).$ Moreover $M$ is diffeomorphic to $ \# k\left(S^{3} \times S^{4}\right) \# \Sigma^{7}$ for some homotopy sphere $\Sigma^{7} \in b P_{8}$ that bounds a parallizable 8-manifold (one of the 28 possible smooth structures on the oriented 7-sphere). Thus, if the link has no torsion it is homeomorphic to $k \#\left(S^{3} \times S^{4}\right)$ where $k$
 is the rank of the third homology group. In \cite{BBG}, it is showed that,  being of Sasaki type, the link has  $k$  even.
 
 Until now, the only cases where Sasaki-Einstein metrics were known to exist on 7-manifolds of the form $\# k \left(S^{3} \times S^{4}\right)$ are $\#222\left(S^{3} \times S^{4}\right)$ and $\#480\left(S^{3} \times S^{4}\right)$.   We found 124 new examples of 2-connected  Sasaki-Einstein 7-manifolds of the form $\#k(S^3\times S^4).$ Of that lot, 118 come from the list given in \cite{JK}, the other 6 examples are links of quasi-smooth Fano 3-folds coming from Reids list of 95 weighted codimension 1 K3 surfaces. We can state the following result.

 \begin{theorem}
 Sasaki-Einstein metrics  exist on 7-manifolds of the form $\#k\left(S^{3} \times S^{4}\right)$ for 22 different values of $k$, where $k=rank(H_{3}(M,\ZZ))$ is given in Table 2 and Table 3. 
 \hfill$\square$
\end{theorem}

Notice that in  Table 2, $k$ assumes every even number between 2 and 32 except $k=22$ and $k=28,$ however the elements in Table 3 give sporadic values for $k.$
\newpage

\begin{center}
\noindent{\bf Table 2: SE  7-manifolds of the form $\#k(S^3\times S^4)$ from the Johnson-Koll\'ar list }
\end{center}
\medskip

\begin{center}
\Small{

\begin{longtable}{| c | c | c | c | c | c | } \hline
${\bf{w}}=(w_0,w_1,w_2, w_3, w_4)$ & Polynomial & Type &   $d$   &  $\mu$  & $H_{3}(M,\ZZ)$   \\ \hline \hline \endfirsthead

\hline
 ${\bf{w}}=(w_0,w_1,w_2, w_3, w_4)$ & Polynomial & Type & $d$   &  $\mu$  & $H_{3}(M,\ZZ)$  \\ \hline \hline \endhead

(9,15,5,10,7)  & $z_{0}^{5}+z_{1}^{3}+z_{1}z_{2}^{6}+z_{2}z_{3}^{4}+z_{3}z_{4}^{5}$ & \tiny{BP + Chain} & 45 & 1216 & $\ZZ^{32}$   \\  \hline

(16,56,7,21,13)  & $z_{0}^{7}+z_{1}^{2}+z_{1}z_{2}^{8}+z_{2}z_{3}^{5}+z_{3}z_{4}^{7}$ & \tiny{BP + Chain} & 112 & 2970 & $\ZZ^{30}$   \\  \hline

(35,15,15,9,32)  & $z_{0}^{3}+z_{1}^{7}+z_{1}z_{2}^{6}+z_{2}z_{3}^{10}+z_{3}z_{4}^{3}$ & \tiny{BP + Chain} & 105 & 1752 & $\ZZ^{24}$   \\  \hline

(50,30,15,9,47)  & $z_{0}^{3}+z_{1}^{5}+z_{1}z_{2}^{8}+z_{2}z_{3}^{15}+z_{3}z_{4}^{3}$ & \tiny{BP + Chain} & 150 & 2472 & $\ZZ^{24}$   \\  \hline

(35,56,63,9,153)  & $z_{0}^{9}+z_{0}z_{1}^{5}+z_{2}^{5}+z_{2}z_{3}^{28}+z_{3}z_{4}^{2}$ & \tiny{Chain + Chain} & 315 & 5328 & $\ZZ^{24}$   \\  \hline

(141,9,138,95,41)  & $z_{0}^{3}+z_{1}^{47}+z_{1}z_{2}^{3}+z_{2}z_{3}^{3}+z_{3}z_{4}^{8}$ & \tiny{BP + Chain} & 423 & 6112 & $\ZZ^{16}$   \\  \hline

(48,10,115,25,43)  & $z_{0}^{5}+z_{1}^{24}+z_{1}z_{2}^{2}+z_{2}z_{3}^{5}+z_{3}z_{4}^{5}$ & \tiny{BP + Chain} & 240 & 3940 & $\ZZ^{20}$   \\  \hline

(49,35,105,10,47)  & $z_{0}^{5}+z_{1}^{7}+z_{1}z_{2}^{2}+z_{2}z_{3}^{14}+z_{3}z_{4}^{5}$ & \tiny{BP + Chain } & 245 & 3168 & $\ZZ^{16}$   \\  \hline

(714,476,119,11,109)  & $z_{0}^{2}+z_{1}^{3}+z_{1}z_{2}^{8}+z_{2}z_{3}^{119}+z_{3}z_{4}^{13}$ & \tiny{BP + Chain } &  1428 & 34294 & $\ZZ^{26}$   \\  \hline

(65,65,39,15,12)  & $z_{0}^{3}+z_{1}^{3}+z_{2}^{5}+z_{3}^{13}+z_{3}z_{4}^{15}$ & \tiny{BP + Chain} & 195 & 2928 & $\ZZ^{24}$   \\  \hline

(119,51,153,12,23)  & $z_{0}^{3}+z_{1}^{7}+z_{1}z_{2}^{2}+z_{2}z_{3}^{17}+z_{3}z_{4}^{15}$ & \tiny{BP + Chain} & 357 & 6680 & $\ZZ^{20}$   \\  \hline

(57,38,95,12,27)  & $z_{0}^{4}+z_{1}^{6}+z_{1}z_{2}^{2}+z_{3}^{19}+z_{3}z_{4}^{8}$ & \tiny{BP + Chain + Chain} & 228 & 2814 & $\ZZ^{18}$   \\  \hline

(81,45,12,131,137)  & $z_{0}^{5}+z_{1}^{9}+z_{1}z_{2}^{30}+z_{2}z_{3}^{3}+z_{3}z_{4}^{2}$ & \tiny{BP + Chain} & 405 & 4288 & $\ZZ^{16}$   \\  \hline

(236,12,87,207,167)  & $z_{0}^{3}+z_{1}^{59}+z_{1}z_{2}^{8}+z_{2}z_{3}^{3}+z_{3}z_{4}^{3}$ & \tiny{BP + Chain} & 708 & 6492 & $\ZZ^{12}$   \\  \hline

(16,14,49,21,13)  & $z_{0}^{7}+z_{1}^{8}+z_{1}z_{2}^{2}+z_{2}z_{3}^{3}+z_{3}z_{4}^{7}$ & \tiny{BP + Chain} & 112 & 1782 & $\ZZ^{18}$   \\  \hline

(35,21,14,13,23)  & $z_{0}^{3}+z_{1}^{5}+z_{1}z_{2}^{6}+z_{2}z_{3}^{7}+z_{3}z_{4}^{4}$ & \tiny{BP + Chain} & 105 & 1312 & $\ZZ^{16}$   \\  \hline

(495,135,736,107,13)  & $z_{0}^{3}+z_{1}^{11}+z_{4}z_{2}^{2}+z_{2}z_{3}^{7}+z_{3}z_{4}^{106}$ & \tiny{BP + Cycle} & 1485 & 29680 & $\ZZ^{20}$   \\  \hline

(17,14,105,75,45)  & $z_{0}^{15}+z_{0}z_{1}^{17}+z_{4}z_{2}^{2}+z_{2}z_{3}^{2}+z_{3}z_{4}^{4}$ & \tiny{Chain + Cycle} & 255 & 3856 & $\ZZ^{16}$   \\  \hline

(133,38,19,14,63)  & $z_{0}^{2}+z_{1}^{7}+z_{2}^{14}+z_{3}^{19}+z_{3}z_{4}^{4}$ & \tiny{BP + Chain} & 266 & 4524 & $\ZZ^{24}$   \\  \hline

(119,21,14,153,51)  & $z_{0}^{3}+z_{1}^{17}+z_{1}z_{2}^{24}+z_{4}z_{3}^{2}+z_{3}z_{4}^{4}$ & \tiny{BP + Chain + Cycle} & 357 & 6272 & $\ZZ^{32}$   \\  \hline

(33,77,77,14,31)  & $z_{0}^{7}+z_{1}^{3}+z_{1}z_{2}^{2}+z_{2}z_{3}^{11}+z_{3}z_{4}^{7}$ & \tiny{BP + Chain} & 231 & 2400 & $\ZZ^{12}$   \\  \hline

(315,90,135,14,77)  & $z_{0}^{2}+z_{1}^{7}+z_{1}z_{2}^{4}+z_{3}^{45}+z_{3}z_{4}^{8}$ & \tiny{BP + Chain + Chain} & 630 & 6952 & $\ZZ^{16}$   \\  \hline

(15,20,85,17,119)  & $z_{0}^{17}+z_{0}z_{1}^{12}+z_{2}^{3}+z_{2}z_{3}^{10}+z_{3}z_{4}^{2}$ & \tiny{Chain + Chain} & 255 & 6016 & $\ZZ^{32}$   \\  \hline

(24,30,15,35,17)  & $z_{0}^{5}+z_{1}^{4}+z_{1}z_{2}^{6}+z_{2}z_{3}^{3}+z_{3}z_{4}^{5}$ & \tiny{BP + Chain} & 120 & 1236 & $\ZZ^{12}$   \\  \hline

(65,39,15,18,59)  & $z_{0}^{3}+z_{1}^{5}+z_{2}^{13}+z_{2}z_{3}^{10}+z_{3}z_{4}^{3}$ & \tiny{BP + Chain} & 195 & 2176 & $\ZZ^{16}$   \\  \hline

(51,15,24,77,89)  & $z_{0}^{5}+z_{1}^{17}+z_{1}z_{2}^{10}+z_{2}z_{3}^{3}+z_{3}z_{4}^{2}$ & \tiny{BP + Chain} & 255 & 2656 & $\ZZ^{16}$   \\  \hline

(143,26,39,15,207)  & $z_{0}^{3}+z_{0}z_{1}^{11}+z_{2}^{11}+z_{2}z_{3}^{26}+z_{3}z_{4}^{2}$ & \tiny{Chain + Chain} & 429 & 9176 & $\ZZ^{32}$   \\  \hline

(215,15,315,33,68)  & $z_{0}^{3}+z_{1}^{43}+z_{1}z_{2}^{2}+z_{2}z_{3}^{10}+z_{3}z_{4}^{9}$ & \tiny{BP + Chain} & 645 & 13848 & $\ZZ^{24}$   \\  \hline

(176,48,15,171,119)  & $z_{0}^{3}+z_{1}^{11}+z_{1}z_{2}^{32}+z_{2}z_{3}^{3}+z_{3}z_{4}^{3}$ & \tiny{BP + Chain} & 528 & 4908 & $\ZZ^{12}$   \\  \hline

(115,69,92,15,55)  & $z_{0}^{3}+z_{1}^{5}+z_{1}z_{2}^{3}+z_{3}^{23}+z_{3}z_{4}^{6}$ & \tiny{BP + Chain + Chain} & 345 & 2552 & $\ZZ^{12}$   \\  \hline

(117,65,104,285,15)  & $z_{0}^{5}+z_{1}^{9}+z_{1}z_{2}^{5}+z_{4}z_{3}^{2}+z_{3}z_{4}^{20}$ & \tiny{BP + Chain + Cycle} & 585 & 5920 & $\ZZ^{16}$   \\  \hline

(395,15,65,474,237)  & $z_{0}^{3}+z_{1}^{79}+z_{1}z_{2}^{18}+z_{4}z_{3}^{2}+z_{3}z_{4}^{3}$ & \tiny{BP + Chain + Cycle} & 1185 & 16128 & $\ZZ^{24}$   \\  \hline

(132,110,275,15,129)  & $z_{0}^{5}+z_{1}^{6}+z_{1}z_{2}^{2}+z_{3}^{44}+z_{3}z_{4}^{5}$ & \tiny{BP + Chain + Chain} & 660 & 4956 & $\ZZ^{12}$   \\  \hline

(228,570,15,125,203)  & $z_{0}^{5}+z_{1}^{2}+z_{1}z_{2}^{38}+z_{2}z_{3}^{9}+z_{3}z_{4}^{5}$ & \tiny{BP + Chain} & 1140 & 11244 & $\ZZ^{12}$   \\  \hline

(655,15,130,367,799)  & $z_{0}^{3}+z_{1}^{131}+z_{1}z_{2}^{15}+z_{2}z_{3}^{5}+z_{3}z_{4}^{2}$ & \tiny{BP + Chain} & 1965 & 23320 & $\ZZ^{20}$   \\  \hline

(740,222,999,15,245)  & $z_{0}^{3}+z_{1}^{10}+z_{1}z_{2}^{2}+z_{3}^{148}+z_{3}z_{4}^{9}$ & \tiny{BP + Chain + Chain} & 2220 & 26070 & $\ZZ^{18}$   \\  \hline

(200,16,127,39,19)  & $z_{0}^{2}+z_{1}^{25}+z_{4}z_{2}^{3}+z_{2}z_{3}^{7}+z_{3}z_{4}^{19}$ & \tiny{BP + Cycle} & 400 & 9576 & $\ZZ^{24}$   \\  \hline

(135,240,765,17,1139)  & $z_{0}^{17}+z_{0}z_{1}^{9}+z_{2}^{3}+z_{2}z_{3}^{90}+z_{3}z_{4}^{2}$ & \tiny{Chain + Chain} & 2295 & 37264 & $\ZZ^{20}$   \\  \hline

(77,33,33,18,71)  & $z_{0}^{3}+z_{1}^{7}+z_{1}z_{2}^{6}+z_{2}z_{3}^{11}+z_{3}z_{4}^{3}$ & \tiny{BP + Chain} & 231 & 1920 & $\ZZ^{12}$   \\  \hline

(4928,896,2921,19,1093)  & $z_{0}^{2}+z_{1}^{11}+z_{4}z_{2}^{3}+z_{2}z_{3}^{365}+z_{3}z_{4}^{9}$ & \tiny{BP + Cycle} & 9856 & 98550 & $\ZZ^{10}$   \\  \hline

(55,22,99,20,25)  & $z_{0}^{4}+z_{1}^{10}+z_{1}z_{2}^{2}+z_{3}^{11}+z_{3}z_{4}^{8}$ & \tiny{BP + Chain + Chain} & 220 & 2574 & $\ZZ^{18}$   \\  \hline

(2990,20,745,349,1877)  & $z_{0}^{2}+z_{1}^{299}+z_{1}z_{2}^{8}+z_{2}z_{3}^{15}+z_{3}z_{4}^{3}$ & \tiny{BP + Chain} & 5980 & 73854 & $\ZZ^{14}$   \\  \hline

(77,21,35,66,33)  & $z_{0}^{3}+z_{1}^{11}+z_{1}z_{2}^{6}+z_{4}z_{3}^{3}+z_{3}z_{4}^{5}$ & \tiny{BP + Chain + Cycle} & 231 & 1680 & $\ZZ^{12}$   \\  \hline

(203,21,294,45,47)  & $z_{0}^{3}+z_{1}^{29}+z_{1}z_{2}^{2}+z_{2}z_{3}^{7}+z_{3}z_{4}^{12}$ & \tiny{BP + Chain} & 609 & 8992 & $\ZZ^{16}$   \\  \hline

(119,119,51,21,48)  & $z_{0}^{3}+z_{1}^{3}+z_{2}^{7}+z_{3}^{17}+z_{3}z_{4}^{7}$ & \tiny{BP + Chain} & 357 & 2472 & $\ZZ^{12}$   \\  \hline

(27,66,207,23,299)  & $z_{0}^{23}+z_{0}z_{1}^{9}+z_{2}^{3}+z_{2}z_{3}^{18}+z_{3}z_{4}^{2}$ & \tiny{Chain + Chain} & 621 & 10360 & $\ZZ^{20}$   \\  \hline

(172,473,1978,23,1311)  & $z_{0}^{23}+z_{0}z_{1}^{8}+z_{2}^{2}+z_{2}z_{3}^{86}+z_{3}z_{4}^{3}$ & \tiny{Chain + Chain} & 3956 & 55890 & $\ZZ^{18}$   \\  \hline

(6615,2835,9472,23,901)  & $z_{0}^{3}+z_{1}^{7}+z_{4}z_{2}^{2}+z_{2}z_{3}^{451}+z_{3}z_{4}^{22}$ & \tiny{BP + Cycle} & 19845 & 238128 & $\ZZ^{12}$   \\  \hline

(28,182,91,39,25)  & $z_{0}^{13}+z_{1}^{2}+z_{1}z_{2}^{2}+z_{2}z_{3}^{7}+z_{3}z_{4}^{13}$ & \tiny{BP + Chain} & 364 & 4068 & $\ZZ^{12}$   \\  \hline

(195,25,190,157,409)  & $z_{0}^{5}+z_{1}^{39}+z_{1}z_{2}^{5}+z_{2}z_{3}^{5}+z_{3}z_{4}^{2}$ & \tiny{BP + Chain} & 975 & 4528 & $\ZZ^{8}$   \\  \hline

(1045,418,190,25,413)  & $z_{0}^{2}+z_{1}^{5}+z_{2}^{11}+z_{2}z_{3}^{76}+z_{3}z_{4}^{5}$ & \tiny{BP + Chain} & 2090 & 13416 & $\ZZ^{8}$   \\  \hline

(9610,620,25,3839,5127)  & $z_{0}^{2}+z_{1}^{31}+z_{1}z_{2}^{744}+z_{2}z_{3}^{5}+z_{3}z_{4}^{3}$ & \tiny{BP + Chain} & 19220 & 253674 & $\ZZ^{18}$   \\  \hline

(4928,896,3277,731,25)  & $z_{0}^{2}+z_{1}^{11}+z_{4}z_{2}^{3}+z_{2}z_{3}^{9}+z_{3}z_{4}^{365}$ & \tiny{BP + Cycle} & 9856 & 98550 & $\ZZ^{10}$   \\  \hline

(170,102,51,27,161)  & $z_{0}^{3}+z_{1}^{5}+z_{1}z_{2}^{8}+z_{2}z_{3}^{17}+z_{3}z_{4}^{3}$ & \tiny{BP + Chain} & 510 & 2792 & $\ZZ^{8}$   \\  \hline


(92,414,207,27,89)  & $z_{0}^{9}+z_{1}^{2}+z_{1}z_{2}^{2}+z_{2}z_{3}^{23}+z_{3}z_{4}^{9}$ & \tiny{BP + Chain} & 828 & 5912 & $\ZZ^{8}$   \\  \hline

(1251,27,414,371,1691)  & $z_{0}^{3}+z_{1}^{139}+z_{1}z_{2}^{9}+z_{2}z_{3}^{9}+z_{3}z_{4}^{2}$ & \tiny{BP + Chain} & 3753 & 24744 & $\ZZ^{12}$   \\  \hline

(28,161,658,47,423)  & $z_{0}^{47}+z_{0}z_{1}^{8}+z_{2}^{2}+z_{2}z_{3}^{14}+z_{3}z_{4}^{3}$ & \tiny{Chain + Chain} & 1316 & 18810 & $\ZZ^{18}$   \\  \hline

(455,105,28,191,587)  & $z_{0}^{3}+z_{1}^{13}+z_{1}z_{2}^{45}+z_{2}z_{3}^{7}+z_{3}z_{4}^{2}$ & \tiny{BP + Chain} & 1365 & 9336 & $\ZZ^{12}$   \\  \hline

(854,28,105,229,493)  & $z_{0}^{2}+z_{1}^{61}+z_{1}z_{2}^{16}+z_{2}z_{3}^{7}+z_{3}z_{4}^{3}$ & \tiny{BP + Chain} & 1708 & 14580 & $\ZZ^{12}$   \\  \hline

(16646,4756,29,1147,10715)  & $z_{0}^{2}+z_{1}^{7}+z_{1}z_{2}^{984}+z_{2}z_{3}^{29}+z_{3}z_{4}^{3}$ & \tiny{BP + Chain} & 33292 & 406386 & $\ZZ^{18}$   \\  \hline

(84,35,77,30,195)  & $z_{0}^{5}+z_{1}^{12}+z_{1}z_{2}^{5}+z_{3}^{14}+z_{3}z_{4}^{2}$ & \tiny{BP + Chain + Chain} & 420 & 2940 & $\ZZ^{12}$   \\  \hline

(405,162,30,65,149)  & $z_{0}^{2}+z_{1}^{5}+z_{2}^{27}+z_{2}z_{3}^{12}+z_{3}z_{4}^{5}$ & \tiny{BP + Chain} & 810 & 5288 & $\ZZ^{8}$   \\  \hline

(1245,830,30,123,263)  & $z_{0}^{2}+z_{1}^{3}+z_{2}^{83}+z_{2}z_{3}^{20}+z_{3}z_{4}^{9}$ & \tiny{BP + Chain} & 2490 & 26724 & $\ZZ^{12}$   \\  \hline

(319,87,58,31,463)  & $z_{0}^{3}+z_{1}^{11}+z_{1}z_{2}^{15}+z_{2}z_{3}^{29}+z_{3}z_{4}^{2}$ & \tiny{BP + Chain} & 957 & 9880 & $\ZZ^{20}$   \\  \hline

(1323,567,1921,32,127)  & $z_{0}^{3}+z_{1}^{7}+z_{4}z_{2}^{2}+z_{2}z_{3}^{64}+z_{3}z_{4}^{31}$ & \tiny{BP + Cycle} & 3969 & 47616 & $\ZZ^{12}$   \\  \hline

(341,93,62,495,33)  & $z_{0}^{3}+z_{1}^{11}+z_{1}z_{2}^{15}+z_{4}z_{3}^{2}+z_{3}z_{4}^{16}$ & \tiny{BP + Chain + Cycle} & 1023 & 9920 & $\ZZ^{20}$   \\  \hline

(1173,782,102,33,257)  & $z_{0}^{2}+z_{1}^{3}+z_{2}^{23}+z_{2}z_{3}^{68}+z_{3}z_{4}^{9}$ & \tiny{BP + Chain} & 2346 & 25068 & $\ZZ^{12}$   \\  \hline

(440,330,33,117,401)  & $z_{0}^{3}+z_{1}^{4}+z_{1}z_{2}^{30}+z_{2}z_{3}^{11}+z_{3}z_{4}^{3}$ & \tiny{BP + Chain} & 1320 & 5514 & $\ZZ^{6}$   \\  \hline

(935,33,308,1275,255)  & $z_{0}^{3}+z_{1}^{85}+z_{1}z_{2}^{9}+z_{4}z_{3}^{2}+z_{3}z_{4}^{6}$ & \tiny{BP + Chain + Cycle} & 2805 & 16344 & $\ZZ^{12}$   \\  \hline

(1700,510,2295,33,563)  & $z_{0}^{3}+z_{1}^{10}+z_{1}z_{2}^{2}+z_{2}z_{3}^{85}+z_{3}z_{4}^{9}$ & \tiny{BP + Chain} & 5100 & 27222 & $\ZZ^{6}$   \\  \hline

(55,77,44,175,35)  & $z_{0}^{7}+z_{1}^{5}+z_{1}z_{2}^{7}+z_{4}z_{3}^{2}+z_{3}z_{4}^{6}$ & \tiny{BP + Chain + Cycle} & 385 & 2232 & $\ZZ^{12}$   \\  \hline

(189,35,182,405,135)  & $z_{0}^{5}+z_{1}^{27}+z_{1}z_{2}^{5}+z_{4}z_{3}^{2}+z_{3}z_{4}^{4}$ & \tiny{BP + Chain + Cycle} & 945 & 3488 & $\ZZ^{8}$   \\  \hline

(324,270,675,35,317)  & $z_{0}^{5}+z_{1}^{6}+z_{1}z_{2}^{2}+z_{2}z_{3}^{27}+z_{3}z_{4}^{5}$ & \tiny{BP + Chain} & 1620 & 5212 & $\ZZ^{4}$   \\  \hline

(6102,36,507,3899,1661)  & $z_{0}^{2}+z_{1}^{339}+z_{1}z_{2}^{24}+z_{2}z_{3}^{3}+z_{3}z_{4}^{5}$ & \tiny{BP + Chain} & 12204 & 105430 & $\ZZ^{10}$   \\  \hline

(345,45,110,37,499)  & $z_{0}^{3}+z_{1}^{23}+z_{1}z_{2}^{9}+z_{2}z_{3}^{25}+z_{3}z_{4}^{2}$ & \tiny{BP + Chain} & 1035 & 10720 & $\ZZ^{12}$   \\  \hline

(315,175,280,37,769)  & $z_{0}^{5}+z_{1}^{9}+z_{1}z_{2}^{5}+z_{2}z_{3}^{35}+z_{3}z_{4}^{2}$ & \tiny{BP + Chain }& 1575 & 6448 & $\ZZ^{8}$   \\  \hline

(27306,18204,37,1475,7591)  & $z_{0}^{2}+z_{1}^{3}+z_{1}z_{2}^{984}+z_{2}z_{3}^{37}+z_{3}z_{4}^{7}$ & \tiny{BP + Chain} & 54612 & 658294 & $\ZZ^{14}$   \\  \hline

(416,624,39,93,77)  & $z_{0}^{3}+z_{1}^{2}+z_{1}z_{2}^{16}+z_{2}z_{3}^{13}+z_{3}z_{4}^{15}$ & \tiny{BP + Chain} & 1248 & 11710 & $\ZZ^{10}$   \\  \hline

(6615,2835,9901,452,43)  & $z_{0}^{3}+z_{1}^{7}+z_{4}z_{2}^{2}+z_{2}z_{3}^{22}+z_{3}z_{4}^{451}$ & \tiny{BP + Cycle} & 19845 & 238128 & $\ZZ^{12}$   \\  \hline

(255,45,60,47,359)  & $z_{0}^{3}+z_{1}^{17}+z_{1}z_{2}^{12}+z_{2}z_{3}^{15}+z_{3}z_{4}^{2}$ & \tiny{BP + Chain} & 765 & 6496 & $\ZZ^{16}$   \\  \hline

(297,81,45,47,422)  & $z_{0}^{3}+z_{1}^{11}+z_{1}z_{2}^{18}+z_{2}z_{3}^{18}+z_{3}z_{4}^{2}$ & \tiny{BP + Chain} & 891 & 7504 & $\ZZ^{12}$   \\  \hline

(464,348,87,45,449)  & $z_{0}^{3}+z_{1}^{4}+z_{1}z_{2}^{12}+z_{2}z_{3}^{29}+z_{3}z_{4}^{3}$ & \tiny{BP + Chain} & 1392 & 5658 & $\ZZ^{6}$   \\  \hline

(387,215,344,945,45)  & $z_{0}^{5}+z_{1}^{9}+z_{1}z_{2}^{5}+z_{4}z_{3}^{2}+z_{3}z_{4}^{22}$ & \tiny{BP + Chain + Cycle} & 1935 & 6512 & $\ZZ^{8}$   \\  \hline

(6600,4400,275,47,1879)  & $z_{0}^{2}+z_{1}^{3}+z_{1}z_{2}^{32}+z_{2}z_{3}^{275}+z_{3}z_{4}^{7}$ & \tiny{BP + Chain} & 13200 & 158494 & $\ZZ^{14}$   \\  \hline

(494,76,57,49,313)  & $z_{0}^{2}+z_{1}^{13}+z_{1}z_{2}^{16}+z_{2}z_{3}^{19}+z_{3}z_{4}^{3}$ & \tiny{BP + Chain} & 988 & 8100 & $\ZZ^{12}$   \\  \hline

(608,114,855,51,197)  & $z_{0}^{3}+z_{1}^{16}+z_{1}z_{2}^{2}+z_{2}z_{3}^{19}+z_{3}z_{4}^{9}$ & \tiny{BP + Chain} & 1824 & 9762 & $\ZZ^{6}$   \\  \hline

(2996,856,107,55,1979)  & $z_{0}^{2}+z_{1}^{7}+z_{1}z_{2}^{48}+z_{2}z_{3}^{107}+z_{3}z_{4}^{3}$ & \tiny{BP + Chain} & 5992 & 72234 & $\ZZ^{18}$   \\  \hline

(34496,9856,22205,59,2377)  & $z_{0}^{2}+z_{1}^{7}+z_{4}z_{2}^{3}+z_{2}z_{3}^{793}+z_{3}z_{4}^{29}$ & \tiny{BP + Cycle} & 68992 & 413946 & $\ZZ^{6}$   \\  \hline

(1323,567,1954,65,61)  & $z_{0}^{3}+z_{1}^{7}+z_{4}z_{2}^{2}+z_{2}z_{3}^{31}+z_{3}z_{4}^{64}$ & \tiny{BP + Cycle} & 3969 & 47616 & $\ZZ^{12}$   \\  \hline

(312,780,65,115,289)  & $z_{0}^{5}+z_{1}^{2}+z_{1}z_{2}^{12}+z_{2}z_{3}^{13}+z_{3}z_{4}^{5}$ & \tiny{BP + Chain} & 1560 & 5084 & $\ZZ^{4}$   \\  \hline

(8960,2560,5793,67,541)  & $z_{0}^{2}+z_{1}^{7}+z_{4}z_{2}^{3}+z_{2}z_{3}^{181}+z_{3}z_{4}^{33}$ & \tiny{BP + Cycle} & 17920 & 107514 & $\ZZ^{6}$   \\  \hline

(4044,2696,1145,131,73)  & $z_{0}^{2}+z_{1}^{3}+z_{4}z_{2}^{7}+z_{2}z_{3}^{53}+z_{3}z_{4}^{109}$ & \tiny{BP + Cycle} & 8088 & 80878 & $\ZZ^{10}$   \\  \hline

(2331,999,74,187,3403)  & $z_{0}^{3}+z_{1}^{7}+z_{1}z_{2}^{81}+z_{2}z_{3}^{37}+z_{3}z_{4}^{2}$ & \tiny{BP + Chain} & 6993 & 43080 & $\ZZ^{12}$   \\  \hline

(6272,1792,4069,75,337)  & $z_{0}^{2}+z_{1}^{7}+z_{4}z_{2}^{3}+z_{2}z_{3}^{113}+z_{3}z_{4}^{37}$ & \tiny{BP + Cycle} & 12544 & 75258 & $\ZZ^{6}$   \\  \hline

(3002,76,741,277,1909)  & $z_{0}^{2}+z_{1}^{79}+z_{1}z_{2}^{8}+z_{2}z_{3}^{19}+z_{3}z_{4}^{3}$ & \tiny{BP + Chain} & 6004 & 24570 & $\ZZ^{6}$   \\  \hline

(3861,81,1278,1145,5219)  & $z_{0}^{3}+z_{1}^{143}+z_{1}z_{2}^{9}+z_{2}z_{3}^{9}+z_{3}z_{4}^{2}$ & \tiny{BP + Chain} & 11583 & 25456 & $\ZZ^{4}$   \\  \hline

(34496,9856,22969,1587,85)  & $z_{0}^{2}+z_{1}^{7}+z_{4}z_{2}^{3}+z_{2}z_{3}^{29}+z_{3}z_{4}^{793}$ & \tiny{BP + Cycle} & 68992 & 413946 & $\ZZ^{6}$   \\  \hline

(5096,1456,3319,87,235)  & $z_{0}^{2}+z_{1}^{7}+z_{4}z_{2}^{3}+z_{2}z_{3}^{79}+z_{3}z_{4}^{43}$ & \tiny{BP + Cycle} & 10192 & 61146 & $\ZZ^{6}$   \\  \hline

(8960,2560,5941,363,97)  & $z_{0}^{2}+z_{1}^{7}+z_{4}z_{2}^{3}+z_{2}z_{3}^{33}+z_{3}z_{4}^{181}$ & \tiny{BP + Cycle} & 17920 & 107514 & $\ZZ^{6}$   \\  \hline

(4150,100,205,1619,2227)  & $z_{0}^{2}+z_{1}^{83}+z_{1}z_{2}^{40}+z_{2}z_{3}^{5}+z_{3}z_{4}^{3}$ & \tiny{BP + Chain} & 8300 & 36438 & $\ZZ^{6}$   \\  \hline

(3535,1515,202,103,5251)  & $z_{0}^{3}+z_{1}^{7}+z_{1}z_{2}^{45}+z_{2}z_{3}^{101}+z_{3}z_{4}^{2}$ & \tiny{BP + Chain} & 10605 & 64248 & $\ZZ^{12}$   \\  \hline

(2666,172,645,109,1741)  & $z_{0}^{2}+z_{1}^{31}+z_{1}z_{2}^{8}+z_{2}z_{3}^{43}+z_{3}z_{4}^{3}$ & \tiny{BP + Chain} & 5332 & 21546 & $\ZZ^{6}$   \\  \hline

(6272,1792,4145,227,109)  & $z_{0}^{2}+z_{1}^{7}+z_{4}z_{2}^{3}+z_{2}z_{3}^{37}+z_{3}z_{4}^{113}$ & \tiny{BP + Cycle} & 12544 & 75258 & $\ZZ^{6}$   \\  \hline

(5096,1456,3355,159,127)  & $z_{0}^{2}+z_{1}^{7}+z_{4}z_{2}^{3}+z_{2}z_{3}^{43}+z_{3}z_{4}^{79}$ & \tiny{BP + Cycle} & 10192 & 61146 & $\ZZ^{6}$   \\  \hline

(10368,6912,1741,131,1585)  & $z_{0}^{2}+z_{1}^{3}+z_{4}z_{2}^{11}+z_{2}z_{3}^{145}+z_{3}z_{4}^{13}$ & \tiny{BP + Cycle} & 20736 & 41470 & $\ZZ^{2}$   \\  \hline

(10368,6912,1873,1451,133)  & $z_{0}^{2}+z_{1}^{3}+z_{4}z_{2}^{11}+z_{2}z_{3}^{13}+z_{3}z_{4}^{145}$ & \tiny{BP + Cycle} & 20736 & 41470 & $\ZZ^{2}$   \\  \hline

(18970,5420,813,137,12601)  & $z_{0}^{2}+z_{1}^{7}+z_{1}z_{2}^{40}+z_{2}z_{3}^{271}+z_{3}z_{4}^{3}$ & \tiny{BP + Chain} & 37940 & 152034 & $\ZZ^{6}$   \\  \hline

(6885,1215,2160,137,10259)  & $z_{0}^{3}+z_{1}^{17}+z_{1}z_{2}^{9}+z_{2}z_{3}^{135}+z_{3}z_{4}^{2}$ & \tiny{BP + Chain} & 20655 & 41584 & $\ZZ^{4}$   \\  \hline

(3240,2160,559,191,331)  & $z_{0}^{2}+z_{1}^{3}+z_{4}z_{2}^{11}+z_{2}z_{3}^{31}+z_{3}z_{4}^{19}$ & \tiny{BP + Cycle} & 6480 & 12958 & $\ZZ^{2}$   \\  \hline

(3240,2160,571,311,199)  & $z_{0}^{2}+z_{1}^{3}+z_{4}z_{2}^{11}+z_{2}z_{3}^{19}+z_{3}z_{4}^{31}$ & \tiny{BP + Cycle} & 6480 & 12958 & $\ZZ^{2}$   \\  \hline

(168480,112320,46837,223,9101)  & $z_{0}^{2}+z_{1}^{3}+z_{4}z_{2}^{7}+z_{2}z_{3}^{1301}+z_{3}z_{4}^{37}$ & \tiny{BP + Cycle} & 336960 & 673918 & $\ZZ^{2}$   \\  \hline

(41472,27648,11561,247,2017)  & $z_{0}^{2}+z_{1}^{3}+z_{4}z_{2}^{7}+z_{2}z_{3}^{289}+z_{3}z_{4}^{41}$ & \tiny{BP + Cycle} & 82944 & 165886 & $\ZZ^{2}$   \\  \hline

(168480,112320,48101,7807,253)  & $z_{0}^{2}+z_{1}^{3}+z_{4}z_{2}^{7}+z_{2}z_{3}^{37}+z_{3}z_{4}^{1301}$ & \tiny{BP + Cycle} & 336960 & 673918 & $\ZZ^{2}$   \\  \hline

(41472,27648,11809,1735,281)  & $z_{0}^{2}+z_{1}^{3}+z_{4}z_{2}^{7}+z_{2}z_{3}^{41}+z_{3}z_{4}^{289}$ & \tiny{BP + Cycle} & 82944 & 165886 & $\ZZ^{2}$   \\  \hline

(24840,16560,6947,283,1051)  & $z_{0}^{2}+z_{1}^{3}+z_{4}z_{2}^{7}+z_{2}z_{3}^{151}+z_{3}z_{4}^{47}$ & \tiny{BP + Cycle} & 49680 & 99358 & $\ZZ^{2}$   \\  \hline

(24840,16560,7051,907,323)  & $z_{0}^{2}+z_{1}^{3}+z_{4}z_{2}^{7}+z_{2}z_{3}^{47}+z_{3}z_{4}^{151}$ & \tiny{BP + Cycle} & 49680 & 99358 & $\ZZ^{2}$   \\  \hline

(18792,12528,5279,355,631)  & $z_{0}^{2}+z_{1}^{3}+z_{4}z_{2}^{7}+z_{2}z_{3}^{91}+z_{3}z_{4}^{59}$ & \tiny{BP + Cycle} & 37584 & 75166 & $\ZZ^{2}$   \\  \hline

(18792,12528,5311,547,407)  & $z_{0}^{2}+z_{1}^{3}+z_{4}z_{2}^{7}+z_{2}z_{3}^{59}+z_{3}z_{4}^{91}$ & \tiny{BP + Cycle} & 37584 & 75166 & $\ZZ^{2}$   \\  \hline

\end{longtable}
}
\end{center}
\medskip

Reid's list of 95 codimension one K3 surfaces of the form 
$Y_d\subset\mathbb{CP}( w_{1}, w_{2}, w_{3}, w_{4})$ with $\sum_{i=1}^{4} w_{i}=d$  can be used to generate $\mathbb{Q}$-Fano 3-folds \cite{IF}. These threefolds are hypersurfaces $X_{d}$ of degree $d$ in weighted projective spaces of the form $\mathbb{CP}(1, w_{1}, w_{2}, w_{3}, w_{4})$ with 
$\sum_{i=1}^{4} w_{i}=d$.  In \cite{Ch}, Cheltsov studied  these sort of $\mathbb{Q}$-Fano 3-folds and  proves that 91 elements  admit K\"ahler-Einstein orbifold metrics and thereby the links associated to them admit Sasaki-Einstein metrics. The four that fail the test are numbers $1,2,4$, and $5$ in the list given in \cite{IF}.  
From this list, we found that 88 of these links are links of singularities of chain type, cycle type or an   
iterated Thom-Sebastiani sum of chain type singularities and cycle type singularities. 
We compute the third homology group for this lot  and found 6 new examples of Sasaki-Einstein 7-manifolds with no torsion on the third homology group. In Table 3 we present these six elements and include the two elements found by Boyer in \cite{Bo}, the ones with weights  $(1,1,1,4,6)$ and $(1,1,6,14, 21)$. 
\medskip

\begin{center}
\noindent{\bf Table 3: SE  7-manifolds of the form $\#k (S^3\times S^4)$ from Fletcher's list}
\end{center}
\medskip

\begin{center}
\begin{longtable}{| c | c | c | c | c | c |} \hline

${\bf{w}}=(w_0,w_1,w_2, w_3, w_4)$ & Polynomial & Type &   $d$   &  $\mu$  & $H_{3}(M,\ZZ)$   \\ \hline \hline \endfirsthead

\hline
 ${\bf{w}}=(w_0,w_1,w_2, w_3, w_4)$ & Polynomial & Type & $d$   &  $\mu$  & $H_{3}(M,\ZZ)$  \\ \hline \hline \endhead


\hline

(1,1,1,4,6) & $z_{0}^{12}+z_{1}^{12}+z_{2}^{12}+z_{3}^{3}+z_{4}^{2}$ & \tiny{BP} & 12 & 2662 & $\ZZ^{222}$ \\  \hline

(7,3,1,10,1) & $z_{0}^{3}+z_{1}^{7}+z_{1}z_{2}^{18}+z_{2}z_{3}^{2}+z_{3}z_{4}^{11}$ & \tiny{BP + Chain} & 21 & 5280 & $\ZZ^{252}$ \\  \hline

(14,4,1,9,1) & $z_{0}^{2}+z_{1}^{7}+z_{1}z_{2}^{24}+z_{2}z_{3}^{3}+z_{3}z_{4}^{19}$ & \tiny{BP + Chain} & 28 & 9234 & $\ZZ^{330}$ \\  \hline

(11,3,5,14,1) & $z_{0}^{3}+z_{1}^{11}+z_{1}z_{2}^{6}+z_{2}z_{3}^{2}+z_{3}z_{4}^{19}$ & \tiny{BP + Chain} & 33 & 4864 & $\ZZ^{148}$ \\  \hline

(18,12,1,5,1) & $z_{0}^{2}+z_{1}^{3}+z_{1}z_{2}^{24}+z_{2}z_{3}^{7}+z_{3}z_{4}^{31}$ & \tiny{BP + Chain} & 36 & 15190 & $\ZZ^{422}$ \\  \hline

(1,1,6,14,21) & $z_{0}^{42}+z_{1}^{42}+z_{2}^{7}+z_{3}^{3}+z_{4}^{2}$ & \tiny{BP} & 42 & 20172 & $\ZZ^{480}$ \\  \hline

(22,4,5,13,1) & $z_{0}^{2}+z_{1}^{11}+z_{1}z_{2}^{8}+z_{2}z_{3}^{3}+z_{3}z_{4}^{31}$ & \tiny{BP + Chain} & 44 & 7998 & $\ZZ^{182}$ \\  \hline

(33,22,6,5,1) & $z_{0}^{2}+z_{1}^{3}+z_{2}^{11}+z_{2}z_{3}^{12}+z_{3}z_{4}^{61}$ & \tiny{BP + Chain} & 66 & 15860 & $\ZZ^{240}$ \\  \hline
\end{longtable}
\end{center}
This list is not necessary exhaustive for Cheltsov's list, since neither numbers 52, 81 nor 86 in \cite{IF} admit descriptions in terms of the  type of singularities where Orlik's conjecture is known to be  valid. However, our computer program suggests that these members do have torsion.

Also, it is interesting to notice that, the $\mathbb{Q}$-Fano 3-folds considered by Cheltsov are $d$-fold branch covers for certain weighted projective spaces branched over orbifold $K3$ surfaces of degree  $d$.   
It follows from a well-known result on links of branched covers (see Proposition 2.1 in \cite{DK}), that the   91 Sasaki-Einstein links associated to the 3-folds of Cheltsov's list can be realized as $d$-fold branched  covers  of $S^7$ branched along the submanifolds $\# k(S^2\times S^3).$ Actually, we extend Theorem 4.6 in \cite{Bo}. 

\begin{theorem}
There exist Sasaki-Einstein metrics on the 7-manifolds $M^7$ which can be realized as $d$-fold branched covers of $S^{7}$ branched along the submanifolds $ \# k(S^{2}\times S^{3})$ where  $k$ ranges from 3 to 21, except $k=17.$ The $\mathbb Q$-Fano 3-folds, the homology of $M^7$ and $k$ are given in Table I in Appendix, $(b)$. 
\end{theorem}  
 \begin{proof}  This  is just a consequence of Theorem A in \cite{Cu} (also see Appendix B in \cite{BBG}),  and the fact that for each $k\not=17$ ranging from 3 to 21 there is an element in  Cheltsov's list that can be given as an hypersurface  singularity that is  of one of the types admissible for Orlik's conjecture.
\end{proof}

In Appendix (c) we also  include Tables II and III which 
list the third homology group of links of hypersurface singularities which are neither rational homology 7-spheres nor homeomorphic to connected sums of $S^3\times S^4$. Table II  consists of 79 links coming from Cheltsov's list and Table III consists of 1503 links  extracted from the Johnson and Koll\'ar list.

\section{appendix}
\begin{itemize}
\item[{\bf (a)}] The Johnson-Kollár list of hypersurfaces in weighted projective 4-space 
admitting Kähler-Einstein orbifold metrics is available at
\href{https://web.math.princeton.edu/~jmjohnso/delpezzo/KEandTiger.txt}{https://web.math.princeton.edu/~jmjohnso/delpezzo/ \newline KEandTiger.txt}. This list includes the weight vectors followed by data on whether or not it is known if the orbifold is Kähler-Einstein.

\item[{\bf (b)}] Table I, cited  in Theorem 3.2 is given here \href{https://github.com/Jcuadrosvalle/TABLES}{https://github.com/Jcuadrosvalle/TABLES}

\item[{\bf (c)}] Table II and Table III, where we exhibit the third homology group of links of hypersurface singularities that are neither rational homology 7-spheres nor homeomorphic to connected sums of $S^3\times S^4,$ are available at \href{https://github.com/Jcuadrosvalle/TABLES}{https://github.com/Jcuadrosvalle/TABLES}

\item[{\bf (d)}] Four codes in Matlab are available at 
\href{https://github.com/Jcuadrosvalle/Codes-in-Matlab}{https://github.com/Jcuadrosvalle/Codes-in-Matlab}. Codes1, 2 and 3  determine whether or not  the  singularities are of chain type, cycle type or an   iterated Thom-Sebastiani sum of chain type  and cycle type singularities, Code4  computes the third homology groups of the corresponding links.

\end{itemize}

\end{document}